\documentclass{amsart}
\usepackage{euscript,amsmath, amssymb, amsfonts, mathtools}
\usepackage{mathrsfs}
\usepackage{mathabx}
\usepackage{stmaryrd}
\usepackage{crossreftools}

\makeatletter
\@namedef{subjclassname@1991}{Subject}
\makeatother

\usepackage[hypertexnames = false, colorlinks=true, allcolors=blue]{hyperref}

\usepackage{tikz-cd}

\numberwithin{equation}{section}
\usepackage{enumitem}   

\newtheorem{theorem}{Theorem}[section]
\newtheorem{lemma}[theorem]{Lemma}
\newtheorem{proposition}[theorem]{Proposition}

\theoremstyle{definition}
\newtheorem{definition}[theorem]{Definition}

\theoremstyle{remark}

\usepackage{graphicx}  
\usepackage{icomma}
\usepackage{bbm}
\usepackage{tikz}
\usetikzlibrary{shapes}
\usepackage{color}
\usepackage{verbatim}
\usepackage{xurl}
\usepackage{indentfirst}
\usepackage{lipsum}
\usepackage[backend = biber]{biblatex}
\usepackage{placeins}
\usepackage{bigints}
\usepackage{subcaption}
\usepackage{appendix}

\DeclareMathOperator{\conv}{conv}

\DeclareMathOperator{\CC}{\mathbb C}
\DeclareMathOperator{\RR}{\mathbb R}
\DeclareMathOperator{\ZZ}{\mathbb Z}

\DeclareMathOperator{\PP}{\mathbb P}
\DeclareMathOperator{\TT}{\mathbb T}
\DeclareMathOperator{\DD}{\mathbb D}

\usepackage{csquotes}

\begin{document}
	\title[Six points can be split by a triple of disjoint discs]{Any six points on the Riemann sphere can be split into three pairs by a triple of disjoint discs}
	
	\author{Matvey Smirnov}
	\address{119991 Russia, Moscow GSP-1, ul. Gubkina 8,
		Institute for Numerical Mathematics,
		Russian Academy of Sciences}
	\email{matsmir98@gmail.com}

	\begin{abstract}
		We prove that for any six points on the Riemann sphere there exist three disjoint closed (or open) discs, each of which contains exactly two of the six distinguished points. This statement shows that recently proposed method to numerically evaluate Kleinian hyperelliptic functions of genus 2 is applicable to any complex curve of genus 2.
		
		\smallskip
		\noindent \textbf{Keywords.} Elementary geometry, analytic geometry, Riemann sphere.
	\end{abstract}
	\subjclass{51M04, 51N20}
	\maketitle
	
	\section{Introduction}
	
	In this paper we denote the Riemann sphere by $\CC\PP(1)$. As usual we identify the complex plane $\CC$ with the affine part of $\CC\PP(1)$, so $\CC\PP(1) = \CC\cup \{\infty\}$. Also let $\mathbb D = \{z \in \CC: |z| < 1\}$ denote the open unit disc, and $\mathbb T = \{z \in \CC: |z| = 1\} = \partial \mathbb D$ denote the unit circle. Recall that biholomorphic mappings $S: \CC\PP(1) \to \CC\PP(1)$ are precisely the M\"obius transformations, i.e. 
	\[
		S(z) = \dfrac{az+ b}{cz + d},
	\]
	where $ad - bc \ne 0$. For the basic properties of Riemann sphere and M\"obius transformations we refer to~\cite[Chap.~VI]{Cartan}.
	
	By a {\it{closed}} (resp. {\it{open}}) {\it{disc}} in $\CC\PP(1)$ we call a set $D \subset \CC\PP(1)$ that is the image of $\overline{\mathbb D}$ (resp. $\mathbb D$) with respect to a M\"obius transformation. That is, a closed disc in $\CC\PP(1)$ is either a closed disc in $\CC$ (with arbitrary center and radius), or the complement of an open disc in $\CC$, or a closed half-plane (with $\infty$). The open discs are described similarly.
	
	The aim of this paper is to prove the following fact.
	
	\begin{theorem}\label{thMain}
		Let $A \subset \CC\PP(1)$ be any six-element set. Then there exist three disjoint closed discs $D_1,D_2,D_3 \subset \CC\PP(1)$ such that $D_j \cap A$ consists of exactly two points for all $j \in \{1,2,3\}$.
	\end{theorem}

	It is easy to see that the following statements are equivalent to Theorem~\ref{thMain}.
	
	\begin{theorem}\label{thMainAlt}
		Let $A \subset \CC\PP(1)$ be any six-element set. Then there exist three disjoint open discs $D_1,D_2,D_3 \subset \CC\PP(1)$ such that $D_j \cap A$ consists of exactly two points for all $j \in \{1,2,3\}$.
	\end{theorem}
	
	\begin{theorem}\label{thMainAltAlt}
		Let $S \subset \RR^3$ be the Euclidean 2-dimensional sphere, i.e. $S = \{(x,y,z) \in \RR^3: x^2 + y^2 + z^2 = 1\}$, and let $A \subset S$ be any six-element set. Then there exist three disjoint open (or closed) discs $B_1,B_2,B_3 \subset S$ (here by discs we mean balls in metric space $S$ equipped with Euclidean distance) such that $B_j \cap A$ consists of exactly two points for all $j \in \{1,2,3\}$.
	\end{theorem}

	The proof of equivalence of Theorems~\ref{thMain}-\ref{thMainAltAlt} is obvious and we omit it.
	
	The motivation for this problem comes from the paper~\cite{FinalAlgorithm}, where such splittings of six points into three pairs were used to formulate the algorithm that computes Kleinian hyperelliptic functions (for definitions see, e.g.~\cite{BuchOldAndNew},~\cite{KleinianWeight2}) associated with complex curves of genus 2. This algorithm exploits the recursive reduction of the problem from a given curve to an isogenous curve, which can be found by a construction due to Richelot (see, e.g.~\cite{Humbert},~\cite[Chap.~8]{Smith}). The construction of an isogenous curve is not unique: there are (in the generic case) 15 options that correspond to partitions of the set of Weierstrass points of the initial curve into three pairs. In~\cite{FinalAlgorithm} it was shown that given three disjoint open discs such that each of them contains exactly two of the Weierstrass points, then the Weierstrass points of corresponding Richelot isogenous curve are again split into three pairs by the same triple of discs. This fact allows one to iterate the Richelot's construction and obtain effective numerical procedures for periods of canonical differentials on a curve of genus 2 and for associated Kleinian functions. However, in~\cite{FinalAlgorithm} it was not proved that such triple of discs always exists.
	
	The rest of the paper is devoted to the proof of Theorem~\ref{thMain}. The structure of this proof (and the sections of this paper) is discussed in the next section. Here we only note that the proof is very elementary and based on it one can quite easy formulate an algorithm that constructs a suitable triple of discs for any six-point set $E \subset \CC\PP(1)$. We shall not formulate such an algorithm.
	
	\section{Basic reduction of the problem and the idea of the proof}
	
	For brevity we call a set $A \subset \CC\PP(1)$ that contains exactly six points {\it{splittable}} if there exist three disjoint closed discs $D_1,D_2,D_3 \subset \CC\PP(1)$ such that $A \cap D_j$ contains exactly two points for all $j \in \{1,2,3\}$. With this terminology we can formulate Theorem~\ref{thMain} as the statement that all six-element sets $A \subset \CC\PP(1)$ are splittable.
	
	\begin{lemma}\label{lemProjectiveInvarianceSplittable}
		Let $A \subset \CC\PP(1)$ be a six-element set and $S:\CC\PP(1) \to \CC\PP(1)$ be a M\"obius transformation. Then $A$ is splittable if and only if $S(A)$ is splittable.
	\end{lemma}
	\begin{proof}
		This is obvious, since M\"obius transformations are bijective and map closed discs to closed discs.
	\end{proof}
	
	Using Lemma~\ref{lemProjectiveInvarianceSplittable} we can without loss of generality impose additional conditions on a six-element set $A$ in order to verify its splittability. To formulate such conditions we consider the set
	\[
	\Omega = \{z \in \CC: |z - 1| < 2\} \cup \{z \in \CC: |z + 1| < 2\} = (1 + 2\mathbb D) \cup (-1 + 2\mathbb D),
	\]
	which is shown on Fig.~\ref{figOmega}.
	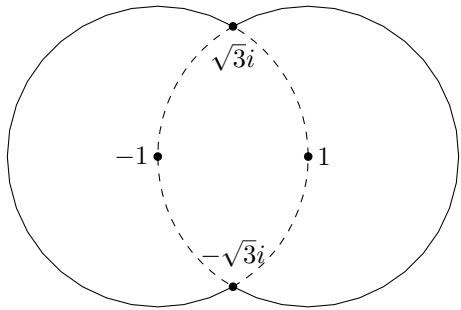
\begin{figure}
	\begin{tikzpicture}
		\draw[black, domain = -120:120] plot({1 + 2*cos(\x)}, {2*sin(\x)});
		\draw[black, dashed, domain = 120:240] plot({1 + 2*cos(\x)}, {2*sin(\x)});
		\draw[black, domain = 60:300] plot({-1 + 2*cos(\x)}, {2*sin(\x)});
		\draw[black, dashed, domain = -60:60] plot({-1 + 2*cos(\x)}, {2*sin(\x)});
		\filldraw (0, {sqrt(3)}) circle[radius = 0.05] node[below, yshift=-3] {$\sqrt{3}i$};
		\filldraw (0, {-sqrt(3)}) circle[radius = 0.05] node[above, yshift=4] {$-\sqrt{3}i$};
		\filldraw (1, 0) circle[radius = 0.05] node[right] {$1$};
		\filldraw (-1, 0) circle[radius = 0.05] node[left] {$-1$};
	\end{tikzpicture}
	\caption{The set $\Omega$; the boundary of $\Omega$ is shown by the solid line.}
	\label{figOmega}
	\end{figure}
	\begin{definition}
		We call a three-element set $E \subset \CC$ distinguished if $E \cap \Omega = \emptyset$ and $|z - w| \ge 2$ for all $z,w \in E$, $z \ne w$.
	\end{definition}
	\begin{proposition}\label{propReductionToSharp}
		Assume that for any distinguished three-element set $E \subset \CC$ the set $E \cup \{1, -1, \infty\} \subset \CC\PP(1)$ is splittable. Then any six-element set $A \subset \CC\PP(1)$ is splittable.
	\end{proposition}
	\begin{proof}
		Let $A \subset \CC\PP(1)$ be arbitrary six-element set. At first we can find a M\"obius transformation $S$ such that $\infty \in S(A)$. After that we can choose distinct $z,w \in S(A)\setminus \{\infty\}$ with smallest possible distance and find an affine transformation $T$, such that $T(z) = -1$ and $T(w) = 1$. Thus, we arrive at the six-element set $T(S(A))$ such that $\{-1, 1, \infty\} \subset T(S(A))$ and $2$ (which is the distance between $-1$ and $1$) is the smallest possible distance between distinct elements of $T(S(A)) \setminus\{\infty\}$. Now it is easy to verify that the set $E = T(S(A)) \setminus \{-1, 1, \infty\}$ is distinguished, so $T(S(A))$ is splittable by assumption (and so is $A$ by Lemma~\ref{lemProjectiveInvarianceSplittable}).
	\end{proof}
	
	Now the problem is to show splittability only for the sets of the form $A = E \cup \{-1,1,\infty\}$, where $E$ is distinguished. Since $A$ constains $\infty$ it will be customary to us to reformulate splittability in this case. From now on we shall denote the convex hull of a set $X \subset \CC$ by $\conv X$.
	
	\begin{lemma}\label{lemFinishingSplitting}
		Let $e_1, \dots, e_5$ be any five distinct points in $\CC$. Assume further that there exist closed bounded disjoint discs $F_1, F_2 \subset \CC$ such that $e_1, e_2 \in F_1$, $e_3, e_4 \in F_2$ and $e_5 \notin \conv(F_1 \cup F_2)$. Then the set $\{e_1, e_2, e_3, e_4, e_5, \infty\}$ is splittable.
	\end{lemma}
	\begin{proof}
		Since $e_5 \notin \conv(F_1 \cup F_2)$ and the set $\conv(F_1 \cup F_2)$ is convex and compact, there exists a closed half-plane $H \subset \CC\PP(1)$ such that $e_5 \in H$ and $H \cap \conv(F_1 \cup F_2) = \emptyset$ (see~\cite[Theorems~1.12 and~1.14]{ConvexLectures}). Clearly, $F_1$, $F_2$, and $H$ constitute a triple of closed disjoint discs that splits $\{e_1, e_2, e_3, e_4, e_5, \infty\}$ into three pairs.
	\end{proof}
	
	As we have finished the foregoing preparation we now discuss the structure of the proof. The main idea is to consider simple methods that may yield a suitable triple of discs for a given distinguished set $E$. These methods will not work for all possible sets $E$, but will sufficiently narrow the remaining configurations. Modulo small subtleties we introduce only two such methods. The first one is to consider the strip $S = \{z \in \mathbb C: |\mathrm{Im}(z)| \le 1\}$. Clearly, if $E \cap S = \emptyset$ and $E$ does not lie entirely on the one side of $S$ (i.e. there are $z$ and $w$ in $E$ with distinct signs of imaginary parts), then $E \cup \{-1,1,\infty\}$ is splittable. Moreover, as we shall see, splittability holds even if $E$ satisfies the foregoing property after some rotation. Also it appears that for distinguished sets $E$ the fact that no rotation of the strip $S$ splits $E$ already significantly restricts possible configurations of $E$. Next method is to try to use discs for which a given pair of points in $E$ constitutes a diameter. In conjunction with the fact that $E$ cannot be splitted by a strip this again severely narrows the remaining cases. After these two steps the proof is basically reduced to considering a concrete construction of discs and verifying several inequalities, that show that discs do not intersect. However, this third step is the most technical one, as it requires a lot of calculations (mostly, analyzing extrema of some elementary functions in order to obtain estimates on them).
	
	The next three sections correspond to the described steps of the proof. In section~\ref{sec:Strips} we obtain some conditions on a distinguished set $E$ assuming it cannot be splitted by a strip. The main result of Section~\ref{sec:Diameters} is the simple sufficient condition for splittability of $E \cup \{-1,1,\infty\}$ for a distinguished set $E$ that cannot be separated by a strip. This condition uses only the discs for which a given pair of points in $E$ constitutes a diameter. Finally, in Section~\ref{sec:Calculations} we consider the distinguished sets $E$, that were not covered by preceding sections, and finalize the proof. In order to improve readability the key steps of the proof are called ``propositions'', while auxiliary statements are called ``lemmas''.
	
	\section{Splitting by a strip}\label{sec:Strips}
	
	\begin{definition}\label{defSplittingByAStrip}
		Let $E \subset \CC$ be a three-element set. We say that $E$ is splittable by a strip (see Fig.~\ref{figSplittableByAStrip}) if, there exists $a \in \TT$ such that $|\mathrm{Im}(z/a)| > 1$ for all $z \in E$, and there exist $z,w \in E$ such that $\mathrm{Im}(z/a) > 1$ and $\mathrm{Im}(w/a) < -1$.
	\end{definition}

	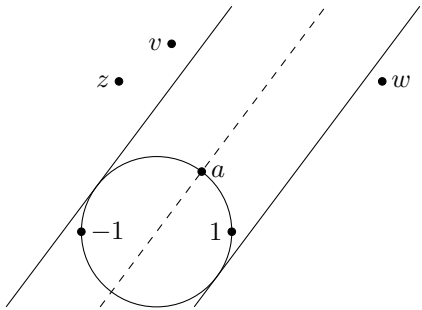
\begin{figure}
		\begin{tikzpicture}
			\draw (0, 0) circle (1);
			\filldraw (1, 0) circle[radius = 0.05] node[left] {$1$};
			\filldraw (-1, 0) circle[radius = 0.05] node[right] {$-1$};
			\draw (0.5, -1) -- (3.5, 3);
			\draw (-2, -1) -- (1, 3);
			\draw[dashed] (-0.75, -1) -- (2.25, 3);
			\filldraw (-0.5, 2) circle[radius = 0.05] node[left] {$z$};
			\filldraw (0.2, 2.5) circle[radius = 0.05] node[left] {$v$};
			\filldraw (3, 2) circle[radius = 0.05] node[right] {$w$};
			\filldraw (0.6, 0.8) circle[radius = 0.05] node[right] {$a$};
		\end{tikzpicture}
		\caption{An example of a set $E = \{z,w,v\}$ splittable by a strip.}
		\label{figSplittableByAStrip}
	\end{figure}
	
	\begin{proposition}\label{propSplittingIfAsterisk}
		Let $E \subset \CC$ be a three-element set that is splittable by a strip. Then the set $\{-1, 1, \infty\} \cup E$ is splittable.
	\end{proposition}
	\begin{proof}
		Let $a \in \CC$ be the number from Definition~\ref{defSplittingByAStrip} applied to $E$, and choose $z, w \in E$ such that $\mathrm{Im}(z/a) > 1$ and $\mathrm{Im}(w/a) < -1$. Clearly, $z \ne w$, so $E = \{z, w, v\}$ for some $v$ (which also satisfies $|\mathrm{Im}(v/a)| > 1$). Assume for now that $\mathrm{Im}(v/a) > 1$. Then, the points $z/a$ and $v/a$ belong to the half-plane $H = \{t \in \CC: \mathrm{Im}(t) > 1\}$, so there exists a closed bounded disc $D_1$, such that $z/a, v/a \in D_1 \subset H$. Also, the points $1/a$ and $-1/a$ belong to the closed unit disc $D_2 = \overline{\mathbb D}$. Clearly, $D_1$ and $D_2$ are disjoint. Moreover, $D_1 \cup D_2$ is contained in the closed half-plane $G = \{t \in \CC: \mathrm{Im}(t) \ge -1\}$, which is a convex set, so $\conv(D_1 \cup D_2) \subset G$. Therefore, $w/a \notin \conv(D_1 \cup D_2)$. By Lemma~\ref{lemFinishingSplitting} the set 
		\[
		A = \left\{\dfrac{1}{a}, -\dfrac{1}{a}, \dfrac{z}{a}, \dfrac{w}{a}, \dfrac{v}{a}, \infty\right\}
		\]
		is splittable. Since $A$ is the image of $\{-1, 1, \infty\} \cup E$ with respect to a M\"obius transformation (namely, scalar multiplication by $a^{-1}$), by Lemma~\ref{lemProjectiveInvarianceSplittable} $\{-1, 1, \infty\} \cup E$ is splittable. The case when $\mathrm{Im}(v/a) < -1$ is handled similarly.
	\end{proof}
	
	The importance of splittability by strips for us lies in the properties of distinguished sets $E$ which we can derive from the assumption that $E$ is {\it{not}} splittable by a strip. In order to formulate these properties let us introduce the half-strip
	\[
	\Sigma = \{z \in \CC: \mathrm{Re}(z) \ge 0 \text{ and } |\mathrm{Im}(z)| \le 1\}.
	\]
	\begin{definition}
		Let us call two complex numbers $z,w \in \CC$ {\it{approximately collinear}}, if there exists $a \in \TT$ such that $z/a, w/a \in \Sigma$.
	\end{definition}
	
	We shall need the following elementary lemma.
	\begin{lemma}\label{lemArcsin}
		Assume that $z \in \CC$, $|z| > 1$, and $t_0 \in \RR$ is chosen such that $z e^{-it_0}$ is a positive real number. Moreover, let $\alpha = \arcsin(1/|z|)$. Then the following statements hold.
		\begin{enumerate}[label=(\roman*)]
			\item\label{arcsini} Let $t \in \RR$.  Then $|\mathrm{Im}(ze^{-it})| \le 1$ if and only if $t \in \bigcup\limits_{n \in \ZZ} [t_0 + \pi n - \alpha, t_0 + \pi n + \alpha]$.
			\item\label{arcsinii} Let $t \in \RR$.  Then $ze^{-it} \in \Sigma$ if and only if $t \in \bigcup\limits_{n \in \ZZ} [t_0 + 2\pi n - \alpha, t_0 + 2\pi n + \alpha]$.
		\end{enumerate}
	\end{lemma}
	\begin{proof}
		Indeed, $|\mathrm{Im}(ze^{-it})| \le 1$ if and only if $|\sin(t - t_0)| \le 1/|z|$. The statement~\ref{arcsini} follows. To prove~\ref{arcsinii} note, that $|\mathrm{Im}(ze^{-it})| \le 1$ and $\mathrm{Re}(ze^{-it}) \ge 0$ if and only if $|\sin(t - t_0)| \le 1/|z|$ and $\cos(t - t_0) \ge 0$.
	\end{proof}

	Before proving the next statement we note that for all $z \in \CC \setminus \Omega$ we have the inequality $|z| \ge \sqrt{3}$. We shall use this fact freely throughout the text.
	
	\begin{lemma}\label{lemIfNotAsteriskThenNoOppositeStrips}
		Assume that $E \subset \CC$ is distinguished and that there exist $w,z \in E$ such that $\conv\{z, w\} \cap \overline{\mathbb D} \ne \emptyset$ (i.e. the line segment from $z$ to $w$ intersects the closed unit disc). Then $E$ is splittable by a strip.
	\end{lemma}
	\begin{proof}
		In order to prove this it suffices to prove that for any distinguished set $E$ there exists $b \in \TT$ such that $|\mathrm{Im}(z/b)| > 1$ for all $z \in E$.
		Indeed, take $z,w \in E$ such that $\conv\{z, w\} \cap \overline{\mathbb D} \ne \emptyset$. If $b$ is chosen as above, then $ \mathrm{Im}(z/b)$ and $\mathrm{Im}(w/b)$ have opposite signs (otherwise the line segment from $z/b$ and $w/b$ would be contained in a half-plane that does not intersect $\overline{\mathbb D}$), so $E$ is splittable by a strip (with $b$ being the corresponding rotation).
		
		In order to prove the foregoing statement for $z \in \CC$ we introduce the set $C_z = \{a \in \TT:, |\mathrm{Im}(z/a)| \le 1\}$. From Lemma~\ref{lemArcsin}~\ref{arcsini} it follows that for $z$ such that $|z| > 1$ the set $C_z$ has angular measure $4\arcsin(1/|z|)$. Now assume that $E$ is a distinguished set such that there is no $b \in \TT$ such that  $|\mathrm{Im}(z/b)| > 1$ for all $z \in E$. Then for any $b \in \TT$ there exists $z \in E$ such that $|\mathrm{Im}(z/b)| \le 1$. In particular, we can apply this to $b_0$ such that $\mathrm{Im}((1 + 2i)/b_0) = 1$ and $\mathrm{Re}((1 + 2i)/b_0) > 0$ (it can be easily seen that $b_0 = (1 + 2i)/(2 + i) = (4 + 3i)/5$). Thus, there exists $z_1 \in E$ such that $|\mathrm{Im}(z_1/b_0)| \le 1$. Since $z_1 \notin \Omega$, it follows that $|z_1| \ge \sqrt{5}$. Similarly (applying previous considerations to $b_0^{-1}$) there exists $z_2 \in E$ such that $|\mathrm{Im}(z_2b_0)| \le 1$. As before, we can conclude that $|z_2| \ge \sqrt{5}$. Moreover, it is straightforward to check that the sets
		\begin{equation}\label{eqAandB}
		A  = \{z \in \CC \setminus \Omega: |\mathrm{Im}(z/b_0)| \le 1\} \text{ and } B  = \{z \in \CC \setminus \Omega: |\mathrm{Im}(zb_0)| \le 1\} 
		\end{equation}
		are disjoint, so $z_1 \ne z_2$ (see Fig.~\ref{figAandBaredisjoint}).
		\begin{figure}
			\begin{tikzpicture}
				\draw[black, dashed, domain = -120:-90] plot({1 + 2*cos(\x)}, {2*sin(\x)});
				\draw[black, domain = -90:{-atan((-4 + 6*sqrt(6))/(3 + 8*sqrt(6)))}] plot({1 + 2*cos(\x)}, {2*sin(\x)});
				\draw[black, dashed, domain = {-atan((-4 + 6*sqrt(6))/(3 + 8*sqrt(6)))}:{atan((-4 + 6*sqrt(6))/(3 + 8*sqrt(6)))}] plot({1 + 2*cos(\x)}, {2*sin(\x)});
				\draw[black, domain = {atan((-4 + 6*sqrt(6))/(3 + 8*sqrt(6)))}:90] plot({1 + 2*cos(\x)}, {2*sin(\x)});
				\draw[black, dashed, domain = 90:120] plot({1 + 2*cos(\x)}, {2*sin(\x)});
				
				\draw[black, dashed, domain = -120:-90] plot({-1 - 2*cos(\x)}, {2*sin(\x)});
				\draw[black, domain = -90:{-atan((-4 + 6*sqrt(6))/(3 + 8*sqrt(6)))}] plot({-1 - 2*cos(\x)}, {2*sin(\x)});
				\draw[black, dashed, domain = {-atan((-4 + 6*sqrt(6))/(3 + 8*sqrt(6)))}:{atan((-4 + 6*sqrt(6))/(3 + 8*sqrt(6)))}] plot({-1 - 2*cos(\x)}, {2*sin(\x)});
				\draw[black, domain = {atan((-4 + 6*sqrt(6))/(3 + 8*sqrt(6)))}:90] plot({-1 - 2*cos(\x)}, {2*sin(\x)});
				\draw[black, dashed, domain = 90:120] plot({-1 - 2*cos(\x)}, {2*sin(\x)});
				
				\draw[dashed] (0,0) circle[radius = 1];
				\filldraw (1, 2) circle[radius = 0.05] node[below right, yshift=-3, xshift = -6] {$1+2i$};
				\filldraw ({4/5}, {3/5}) circle[radius = 0.05] node[above right] {$b_0$};
				\draw ({-17/3}, -3) -- ({-(31 + 16*sqrt(6))/25}, {-(-8 + 12*sqrt(6))/25});
				\draw[dashed] ({-(31 + 16*sqrt(6))/25}, {-(-8 + 12*sqrt(6))/25}) -- (1,2);
				\draw (1,2) -- ({7/3},3);
				
				\draw ({-17/3}, 3) -- ({-(31 + 16*sqrt(6))/25}, {(-8 + 12*sqrt(6))/25});
				\draw[dashed] ({-(31 + 16*sqrt(6))/25}, {(-8 + 12*sqrt(6))/25}) -- (1,-2);
				\draw (1,-2) -- ({7/3},-3);

				\draw ({17/3}, 3) -- ({(31 + 16*sqrt(6))/25}, {(-8 + 12*sqrt(6))/25});
				\draw[dashed] ({(31 + 16*sqrt(6))/25}, {(-8 + 12*sqrt(6))/25}) -- (-1,-2);
				\draw (-1, -2) -- ({-7/3},-3);
				
				\draw ({17/3}, -3) -- ({(31 + 16*sqrt(6))/25}, {-(-8 + 12*sqrt(6))/25});
				\draw[dashed] ({(31 + 16*sqrt(6))/25}, {-(-8 + 12*sqrt(6))/25}) -- (-1,2);
				\draw (-1, 2) -- ({-7/3},3);
				
				\node at (2.5, 2) {$A$};
				\node at (-2.5, -2) {$A$};
				\node at (2.5, -2) {$B$};
				\node at (-2.5, 2) {$B$};
			\end{tikzpicture}
			\caption{The sets $A$ and $B$ from~\eqref{eqAandB}.}
			\label{figAandBaredisjoint}
		\end{figure}
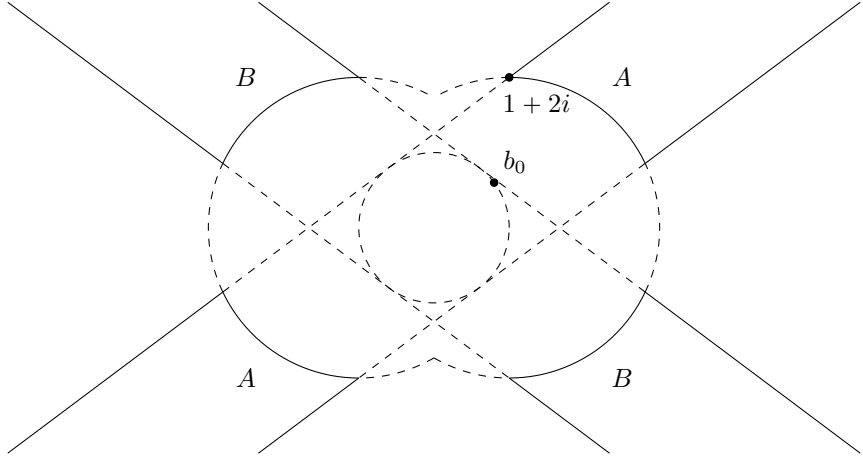 
		Finally, for the third point $z_3 \in E$ we just use the estimate $|z_3| \ge \sqrt{3}$, since $z_3 \notin \Omega$. Therefore, the angular measure of the set $C_{z_1} \cup C_{z_2} \cup C_{z_3}$ can be estimated from above by $\alpha = 4\arcsin(1/\sqrt{3}) + 8\arcsin(1/\sqrt{5})$. Since $\alpha < 2\pi$, we can conclude that there exists $b \in \TT$ such that $b \notin C_{z_1} \cup C_{z_2} \cup C_{z_3}$. This $b$ satisfies $|\mathrm{Im}(z/b)| > 1$ for all $z \in E$.
	\end{proof}
	
	\begin{proposition}\label{propIfSharpNotAsteriskThenStrips}
		Assume that $E \subset \CC$ is a distinguished set that is not splittable by a strip. Then for any $z \in E$ there exists $w \in E \setminus \{z\}$ such that $z$ and $w$ are approximately collinear.
	\end{proposition}
	\begin{proof}
		Assume that there is $z \in E$ such that $z$ and $w$ are not approximately collinear for all $w \in E \setminus \{z\}$. Consider $\gamma = \{a \in \TT: z/a \in \Sigma\}$. It is easy to see that $\gamma$ is a {\it{connected}} closed subset of the unit circle (see Lemma~\ref{lemArcsin}~\ref{arcsinii}). We claim that for all $w \in E\setminus \{z\}$ and $a \in \gamma$ the inequality $|\mathrm{Im}(w/a)| > 1$ holds. Indeed, if $|\mathrm{Im}(w/a)| \le 1$, then either $\mathrm{Re}(w/a) \ge 0$ (which means that $z$ and $w$ are approximately collinear), or $\mathrm{Re}(w/a) \le 0$ (which means that $E$ is splittable by a strip in view of Lemma~\ref{lemIfNotAsteriskThenNoOppositeStrips}). Now let $\{w, v\} = E\setminus \{z\}$. Then either $\mathrm{Im}(w/a) > 1$ for all $a \in \gamma$, or $\mathrm{Im}(w/a) < -1$ for all $a \in \gamma$, since $\gamma$ is connected. Assume that the inequality $\mathrm{Im}(w/a) > 1$ holds for all $a \in \gamma$. Then consider $b \in \gamma$ such that $\mathrm{Im}(z/b) = -1$. Then it is possible to choose small enough $\varepsilon > 0$ such that $\mathrm{Im}(z/(be^{i\varepsilon})) < -1$, $\mathrm{Im}(w/(be^{i\varepsilon})) > 1$, $|\mathrm{Im}(v/(be^{i\varepsilon}))| > 1$. So, $E$ is splittable by a strip with $be^{i\varepsilon}$ being the corresponding rotation. The case when $\mathrm{Im}(w/a) < -1$ holds for all $a \in \gamma$ is treated similarly (by small perturbation of $b \in \gamma$ such that $\mathrm{Im}(z/b) = 1$).
	\end{proof}
	
	Proposition~\ref{propIfSharpNotAsteriskThenStrips} already significantly restricts the positions of points in distinguished three-element sets $E$ for which it is not yet proved that $E \cup \{-1, 1, \infty\}$ is splittable. However, we shall need a more precise description of distinguished sets that are not splittable by a strip.
	
	\begin{proposition}\label{propIfSharpNotAsteriskThenCases}
		Assume that $E \subset \CC$ is a distinguished set that is not splittable by a strip. Then one of the following statements holds.
		\begin{enumerate}[label=(\roman*)]
			\item\label{casei} There exists $a \in \TT$ such that $z/a \in \Sigma$ for all $z \in E$.
			\item\label{caseii} There exists $a\in \TT$ and some enumeration $E = \{z_1, z_2, z_3\}$ such that the following statements hold.
			\begin{enumerate}[label=(\Roman*)]
				\item\label{I} $\mathrm{Re}(z_1/a) \ge 0$ and $\mathrm{Im}(z_1/a) = 1$ (in particular, $z_1/a \in \Sigma$).
				\item\label{II} $z_2 /a \in \Sigma$.
				\item\label{III} $\mathrm{Im}(z_3/a) < -1$.
				\item\label{IV} $z_2$ and $z_3$ are approximately collinear.
			\end{enumerate}
		\end{enumerate}
	\end{proposition}
	\begin{proof}
		
		In order to prove this we need some auxiliary statements. Assume that $p,q \in \CC$ are approximately collinear and $|p|, |q| \ge \sqrt{3}$. Choose $t_0 \in \RR$ such that $e^{-it_0} p$ is a positive real number and let $\alpha = \arcsin(1/|p|)$. Then the following statements hold.
		\begin{enumerate}[label=(\alph*)]
			\item\label{auxa} The set 
			\[I = \{t \in [t_0 - \alpha, t_0 + \alpha]: |\mathrm{Im}(qe^{-it})| \le 1\} \]
			is nonempty closed and connected (i.e. it is a nonempty closed interval). Moreover, $\mathrm{Re}(qe^{-it}) \ge 0$ for all $t \in I$ (that is, $qe^{-it} \in \Sigma$ for $t \in I$).
			\item\label{auxb} Let $s \in [t_0 - \alpha, t_0 + \alpha] \setminus I$. If $s < t$ for some $t \in I$, then $\mathrm{Im}(qe^{-is}) > 1$, and if $s > t$ for some $t \in I$, then $\mathrm{Im}(qe^{-is}) < -1$.
		\end{enumerate}
		To prove~\ref{auxa} note that $I$ is, obviously, closed. To prove that $I$ is connected note that the set $\tilde I = \{t \in \RR: |\mathrm{Im}(qe^{-it})| \le 1\}$ can be (by Lemma~\ref{lemArcsin}~\ref{arcsini}) represented as
		\[
		\tilde I = \bigcup_{n \in \ZZ} [\pi n + s_0 - \beta, \pi n + s_0 + \beta],
		\]
		where $s_0$ is chosen such that $qe^{-is_0}$ is a positive real number, and $\beta = \arcsin(1/|q|)$. Since $I = [t_0 - \alpha, t_0 + \alpha] \cap \tilde I$, the set $I$ can be disconnected only if the interval $[t_0 - \alpha, t_0 + \alpha]$ intersects at least two of the intervals $[\pi n + s_0 - \beta, \pi n + s_0 + \beta]$, $n \in \ZZ$. But this would imply that $2 \alpha + 2\beta \ge \pi$, so $\arcsin(1/\sqrt{3})\ge \pi/4$, which is not true. Thus, $I$ is connected. Moreover, since $|q| \ge \sqrt{3}$ the inequality $|\mathrm{Im}(qe^{-it})| \le 1$ implies that $\mathrm{Re}(qe^{-it}) \ne 0$. Since $I$ is connected, either $\mathrm{Re}(qe^{-it}) > 0$ for all $t \in I$, or $\mathrm{Re}(qe^{-it}) < 0$ for all $t \in I$. Note, however, that since $p$ and $q$ are approximately collinear (in view of Lemma~\ref{lemArcsin}~\ref{arcsinii}) there exists $t \in [t_0 - \alpha, t_0 + \alpha]$ such that $|\mathrm{Im}(qe^{-it})| \le 1$ and $\mathrm{Re}(qe^{-it}) \ge 0$. Thus, $I$ is nonempty and the inequality $\mathrm{Re}(qe^{-it}) \ge 0$ holds for all $t \in I$. That is, we proved~\ref{auxa}. Finally, to prove~\ref{auxb} let 
		\[
		\begin{gathered}
			A = \{s \in [t_0 - \alpha, t_0 + \alpha] \setminus I: \exists t \in I \text{ such that } s < t\}, \\ B = \{s \in [t_0 - \alpha, t_0 + \alpha] \setminus I: \exists t \in I \text{ such that } t < s\}.
		\end{gathered}
		\]
		Clearly, $A$ and $B$ are disjoint, connected, and $[t_0 -\alpha, t_0 + \alpha] \setminus I = A \cup B$ (note that these sets can be empty). Assume that $A$ is nonempty. Since $A$ is connected, either $\mathrm{Im}(qe^{-is}) > 1$ for all $s \in A$, or $\mathrm{Im}(qe^{-is}) < -1$ for all $s \in A$. Let $t = \inf I$. Since the function $s \mapsto \mathrm{Im}(qe^{-is})$ is decreasing on $I$, it is clear that $\mathrm{Im}(qe^{-it}) = 1$, so for $s \in A$ we have $\mathrm{Im}(qe^{-is}) > 1$. The fact that for $s \in B$ we have $\mathrm{Im}(qe^{-is}) < -1$ is proved similarly.
		
		Now we prove the statement of Proposition~\ref{propIfSharpNotAsteriskThenCases}. Due to Proposition~\ref{propIfSharpNotAsteriskThenStrips} we can enumerate the elements of $E$ as $E = \{z_1, z_2, z_3\}$, such that $z_1$ and $z_2$ are approximately collinear, and $z_2$ and $z_3$ are approximately collinear.  Let $t_0$ be chosen such that $z_2 e^{-it_0}$ is a positive real number, and let $\alpha = \arcsin(1/|z_2|)$. Moreover, let
		\[
		\begin{gathered}
			I_1 = \{t \in [t_0 - \alpha, t_0 + \alpha]:  |\mathrm{Im}(z_1 e^{-it})| \le 1\},\\	I_3 = \{t \in [t_0 - \alpha, t_0 + \alpha]:  |\mathrm{Im}(z_3 e^{-it})| \le 1\}.
		\end{gathered}
		\]
		From~\ref{auxa} we see that $I_1$ and $I_3$ are nonempty closed intervals. If $I_1 \cap I_2 \ne \emptyset$, then by choosing $t \in I_1 \cap I_3$ we obtain that statement~\ref{casei} holds with $a = e^{it}$. On the other hand, if $I_1 \cap I_3 = \emptyset$, then (by exchanging $z_1$ and $z_3$ if needed) we can assume that $s < r$ for all $s \in I_3$ and $r \in I_1$. Let $t = \inf I_1$ and $a = e^{it}$. Then with this choice of $a$ we have $\mathrm{Im}(z_1/a) = 1$ and $\mathrm{Im}(z_3/a) < -1$ by~\ref{auxb}. Thus, with this choice of $a$ we obtain properties~\ref{I},~\ref{II}, and~\ref{III} of~\ref{caseii}. Moreover, $z_2$ and $z_3$ are approximately collinear by the choice of ordering, so~\ref{IV} holds.
		
	\end{proof}
	
	We conclude the section with the following lemma, which may be used in conjunction with the case~\ref{caseii} of Proposition~\ref{propIfSharpNotAsteriskThenCases}, but also will be applied on its own. It states, in particular, that given the case~\ref{caseii} of Proposition~\ref{propIfSharpNotAsteriskThenCases} the points $z_1$ and $z_3$ are not approximately collinear.
	
	\begin{lemma}\label{lemIfLargeImaginaryOppositeSignsThenNotApprox}
		Assume that $z_1, z_2 \in \CC$, $\mathrm{Im}(z_1) \ge 1$ and $\mathrm{Im}(z_2) \le -1$. If at least one of these inequalities is strict (i.e. $\max\{|\mathrm{Im}(z_j)|, j = 1,2\} > 1$), then $z_1$ and $z_2$ are not approximately collinear.
	\end{lemma}
	\begin{proof}
		We shall use the following simple fact: if $b \in \TT$ and $\mathrm{Im}(b) > 0$, then $b\Sigma \subset \{z \in \CC: \mathrm{Im}(z) > -1\}$ (and, similarly, if $\mathrm{Im}(b) < 0$, then $b\Sigma \subset \{z \in \CC: \mathrm{Im}(z) < 1\}$). To prove this consider $b \in \TT$ such that $\mathrm{Im}(b) > 0$ and take arbitrary $z \in \Sigma$. Then $\mathrm{Im}(zb) = \mathrm{Re}(z)\mathrm{Im(b)} + \mathrm{Im}(z)\mathrm{Re}(b)$. Clearly, $\mathrm{Re}(z)\mathrm{Im(b)} \ge 0$. Moreover, $|\mathrm{Im}(z)\mathrm{Re}(b)| < 1$, since $|\mathrm{Re}(b)| < 1$ and $|\mathrm{Re}(z)| \le 1$. Thus, $\mathrm{Im}(zb) > -1$. 
		
		Now we return to the proof. Assume that $z_1$ and $z_2$ are approximately collinear. Since $\mathrm{Im}(z_2) \le -1$ the condition $z_2/b \in \Sigma$ implies $\mathrm{Im}(b) \le 0$. Similarly, since $\mathrm{Im}(z_1) \ge 1$, if $z_1/b \in \Sigma$, then $\mathrm{Im}(b) \ge 0$. Thus, if $b \in \TT$ and $z_1/b, z_2/b \in \Sigma$, then $b = \pm 1$. But for $b = \pm 1$ clearly $z_1/b$ and $z_2/b$ cannot simultaneously belong to $\Sigma$, as $\max\{|\mathrm{Im}(z_j)|, j = 1,2\} > 1$.
	\end{proof}
	
	\section{Splitting using discs with diameter on a pair of points from \texorpdfstring{$E$}{E}}\label{sec:Diameters}	
	
	The next step is to eliminate cases, in which we can choose two distinct points $z, w \in E$ such that $F(z,w) \cap \overline{\mathbb D} = \emptyset$ and the third point $v \in E$ is not contained in $\conv(\overline{D} \cup F(z,w))$, where $F(z,w)$ is the closed disc, for which $z$ and $w$ constitute a diameter (that is, they are antipodal boundary points), i.e.
	\[
	F(z,w) = \dfrac{z+ w}{2} + \dfrac{|z - w|}{2} \overline{\mathbb D}.
	\]
	
	To formulate the main result of this section we introduce the following condition on a three-element set $E \subset \CC$.
	\begin{definition}
		We say that a three-element set $E \subset \CC$ is well-separated from zero, if $F(z,w) \cap \overline{\mathbb D} = \emptyset$ for any approximately collinear $z,w \in E$.
	\end{definition}
	
	\begin{proposition}\label{propIfWellSeparatedThenSplittable}
		Let $E$ be a distinguished set that is well-separated from zero. Then $E \cup \{-1,1,\infty\}$ is splittable.
	\end{proposition}
	
	In order to prove this we need several facts about discs $F(z,w)$. At first we get rid of the necessity to consider convex hulls of two discs. To do so we need the following lemma (in this section we shall use only the first statement, while the second one will become useful later).
	
	\begin{lemma}\label{lemConvexHullDescription}
		Let $c_1, c_2 \in \CC$ and $r_1, r_2 > 0$ and let $F_j = c_j + r_j \overline{\mathbb D}$. Then the following statements hold.
		\begin{enumerate}[label=(\roman*)]
			\item\label{convexHulli} Let $c(t) = tc_1 + (1-t)c_2$ and $r(t) = tr_1 + (1-t)r_2$. Then $\conv(F_1 \cup F_2) = \bigcup_{t \in [0,1]} \left(c(t) + r(t)\overline{\mathbb D}\right)$.
			\item\label{convexHullii} For the boundary $\partial \conv(F_1 \cup F_2)$ we have the inclusion
			\begin{multline*}
				\partial \conv(F_1 \cup F_2) \subset \\  F_1 \cup F_2 \cup \{z = c(t) + r(t)a: t \in (0,1), |a| = 1, \mathrm{Re}(a(c_1 - c_2)) = r_2 - r_1\}.
			\end{multline*}
		\end{enumerate}
	\end{lemma}
	\begin{proof}
		The statement~\ref{convexHulli} is a particular case of a more general fact. That is, let $E$ be a real vector space and let $A \subset E$ be convex. Also consider any $v_1, v_2 \in E$ and $r_1, r_2 > 0$. Then, if we let we have the equality
		\begin{equation}\label{eqConvexHullGeneral}
			\conv((v_1 + r_1 A) \cup (v_2 + r_2A)) = \bigcup\limits_{t \in [0,1]} (tv_1 + (1-t)v_2 + (tr_1 + (1-t)r_2)A).
		\end{equation}
		Clearly,~\eqref{eqConvexHullGeneral} would follow from the fact that the set on the right-hand side of~\eqref{eqConvexHullGeneral} is convex, which is an easy consequence of the convexity of $A$. Thus,~\eqref{eqConvexHullGeneral} holds, and so does~\ref{convexHulli}.
		
		To prove~\ref{convexHullii} let $z \in \partial \conv(F_1 \cup F_2) \setminus (F_1 \cup F_2)$. Clearly, from~\ref{convexHulli} this means that $z \in c(t_0) + r(t_0)\overline{\mathbb D}$ for some $t_0 \in (0,1)$. So, we can write $z = c(t_0) + r(t_0)a$, where $|a| \le 1$. If $|a| < 1$, then $z$ is, obviously, an interior point of $\conv(F_1 \cup F_2)$, so $|a| = 1$. Moreover, by the same argument $|z - c(t)| \ge r(t)$ for all $t \in (0,1)$. Thus, the function $f(t) = |z - c(t)|^2 - r(t)^2$ is nonnegative on the interval $(0,1)$ and $f(t_0)= 0$. Therefore, $f'(t_0) = 0$. By expanding the definition of $f$ we obtain that
		\[
		f'(t_0) = 2(\mathrm{Re}(a(c_1 - c_2)) + (r_1 - r_2))(t_0 r_1 + (1-t_0)r_2).
		\]
		Clearly, $t_0 r_1 + (1-t_0)r_2 > 0$, so $f'(t_0) = 0$ implies that $\mathrm{Re}(a(c_1 - c_2)) = (r_2 - r_1)$.
	\end{proof}
	
	\begin{lemma}\label{lemBiggestDoesNotBelongToConvexHull}
		Assume that $z_1,z_2,z_3 \in \CC$, $|z_1 - z_2| \ge 2$, and $|z_3| \ge \max\{|z_1|,|z_2|\}$. Then $z_3 \notin \conv(F(z_1,z_2) \cup \overline{\mathbb D})$ provided $z_3 \notin F(z_1,z_2)$.
	\end{lemma}
	\begin{proof}
		Assume that $z_3 \notin F(z_1,z_2)$. We shall prove a slightly stronger statement, namely, that $z_3 \notin \conv(F(z_1,z_2) \cup r\overline{\mathbb D})$, where $r = |z_1 - z_2|/2 \ge 1$. Let $c = (z_1 + z_2)/2$, i.e. $c$ is the center of $F(z_1,z_2)$. Without loss of generality we may assume that $c$ is a positive real number (if $c$ zero, then the statement is trivial, otherwise, we can apply a suitable rotation). By Lemma~\ref{lemConvexHullDescription}~\ref{convexHulli} we have the equality
		\[
		\conv(F(z_1,z_2) \cup r\overline{\mathbb D}) = \{z \in \CC: \exists \;t \in [0,1] \text{ such that } |z - tc| \le r\}.
		\]
		Now assume that $z_3 \in \conv(F(z_1,z_2) \cup r\overline{\mathbb D})$, so it is possible to represent $z_3$ in the form $z_3 = (\alpha + tc) + i\beta$, where $t \in [0,1]$, $\alpha, \beta \in \RR$, and $\alpha^2 + \beta^2 \le r^2$. From the inequality $|z_3| \ge \max\{|z_1|,|z_2|\}$ it is easy to conclude that $|z_3|^2 \ge c^2 + r^2$, so
		$(\alpha + tc)^2 + \beta^2 \ge c^2 + r^2$ (in particular, $t\ne 0$). Therefore, $|\alpha + tc| \ge c$. If $\alpha + tc \le -c$, then $\alpha$ is negative, so $(\alpha + tc)^2 + \beta^2 < \alpha^2 + t^2c^2 + \beta^2 \le c^2 + r^2$, which contradicts the inequality above. Thus, $\alpha + tc \ge c$. Then, $|z_3 - c|^2 = (\alpha + tc - c)^2 + \beta^2 \le \alpha^2 + \beta^2 \le r$, so $z_3 \in F(z_1, z_2)$.
	\end{proof}
	
	Now, to proceed further, we need some conditions on a triple of points $z,w,v \in \CC$ that imply $v \notin F(z,w)$. Before stating these conditions we need a couple of elementary facts.
	
	\begin{lemma}\label{lemComparisonInStrip}
		Let $z_1, z_2 \in \Sigma$. Assume that $|z_1 - z_2| \ge 2$. Then $\mathrm{Re}(z_1) \le \mathrm{Re}(z_2)$ if and only if $|z_1| \le |z_2|$.
	\end{lemma}
	\begin{proof}
		Assume that $\mathrm{Re}(z_1) \le \mathrm{Re}(z_2)$ and $|z_1| > |z_2|$. Write $z_1 = \alpha + i\beta$ and $z_2 = \gamma + i\delta$. Then $\alpha^2 + \beta^2 > \gamma^2 + \delta^2$, so $\gamma^2 - \alpha^2 < \beta^2 - \delta^2$. Thus, 
		\begin{multline*}
			|z_2 - z_1|^2 = (\gamma - \alpha)^2 + (\delta - \beta)^2 \le (\gamma - \alpha)(\gamma + \alpha) + (\delta - \beta)^2 < \\   \beta^2 - \delta^2 +  \delta^2+ \beta^2 - 2\beta \delta = 2\beta(\delta - \beta) .
		\end{multline*}
		Since $|\beta|, |\delta| \le 1$, it follows that $2\beta(\delta - \beta) \le 4$, so $|z_2 - z_1|^2 < 4$. We have arrived at a contradiction. Thus, the inequality $\mathrm{Re}(z_1) \le \mathrm{Re}(z_2)$ implies $|z_1| \le |z_2|$.
		
		Similar calculation shows the converse (to reuse the calculations of the previous paragraph we exchange the roles of $z_1$ and $z_2$). That is, assume that $\mathrm{Re}(z_1) < \mathrm{Re}(z_2)$ and $|z_1| \ge |z_2|$. Then as above we can prove that $|z_2 - z_1|^2 \le 2\beta(\delta - \beta)$ (note that the inequality now is not strict). Since $|z_2 - z_1|^2 \ge 4$ we can conclude that $\delta = -\beta = \pm 1$. In particular, $\beta^2 = \delta^2$, so inequalities $0 \le \alpha < \gamma$ imply
		\[
		|z_1|^2 = \alpha^2 + \beta^2 < \gamma^2 + \delta^2 = |z_2|^2.
		\]
		This contradiction shows that the inequality $|z_1| \ge |z_2|$ implies $\mathrm{Re}(z_1) \ge \mathrm{Re}(z_2)$.
	\end{proof}
	
	\begin{lemma}\label{lemEqualImagPartsInStrip}
		Assume that $z_1,z_2 \in \CC$ are approximately collinear and $|z_1|, |z_2| > 1$. Then there exists $b \in \TT$ such that $z_1/b, z_2/b \in \Sigma$, and $\mathrm{Im}(z_1/b) = -\mathrm{Im}(z_2/b)$.
	\end{lemma}
	\begin{proof}
		Let $C = \{a \in \TT: z_1/a \in \Sigma, z_2/a \in \Sigma\}$. Clearly, $C$ is a closed nonempty subset of the unit circle. Define $f:C \to \RR$ by the rule 
		\[
		f(a) = \max\{|\mathrm{Im}(z_1/a)|, |\mathrm{Im}(z_2/a)|\}.
		\]
		Since $f$ is continuous and $C$ is compact, there exists $b \in C$ such that $f(b) = \inf \{f(a): a \in C\}$. We claim that $\mathrm{Im}(z_1/b) = -\mathrm{Im}(z_2/b)$. At first assume that $|\mathrm{Im}(z_1/b)| > |\mathrm{Im}(z_2/b)|$. Then it is possible to find arbitrarily small $t \in \RR$ such that $|\mathrm{Im}(e^{-it}z_1/b)|<|\mathrm{Im}(z_1/b)|$. If $t$ is chosen small enough, then still $ e^{-it}z_1/b$ and $e^{-it}z_2/b$ belong to $\Sigma$ (note that $\mathrm{Re}(z_j/b) > 0$, since $|z_j| > 1$), and $|\mathrm{Im}(e^{-it}z_1/b)|>|\mathrm{Im}(e^{-it}z_1/b)|$. Thus, $be^{it} \in C$ and $f(be^{it}) < f(b)$, contradicting the choice of $b$. Therefore, $|\mathrm{Im}(z_1/b)| = |\mathrm{Im}(z_2/b)|$. However, if $\mathrm{Im}(z_1/b) = \mathrm{Im}(z_2/b) \ne 0$, then by a similar argument as above one shows that $b$ is not optimal (by multiplying $b$ with $e^{it}$ for small $t$; it suffices to note that the function $t \mapsto \mathrm{Im}(e^{-it}z)$ monotonically decreases, if $t$ is restricted to the set, where $\mathrm{Re}(e^{-it}z) > 0$). So, $\mathrm{Im}(z_1/b) = -\mathrm{Im}(z_2/b)$.
	\end{proof}
	
	\begin{lemma}\label{lemCriterionForDiscOOnDiameter}
		Let $z_1,z_2,z_3 \in \CC$. Then $z_3 \in F(z_1,z_2)$ if and only if \[\mathrm{Re}\left((z_3 - z_1)(\overline{z_3} - \overline{z_2})\right) > 0.\]
	\end{lemma}
	\begin{proof}
		Indeed, both of the statements are invariant with respect to affine transformations, so we can assume $z_3 = 0$ and $z_1 = 1$ (if $z_3 = z_1$, then the statement is trivial). Clearly, 
		\[
		0 \notin F(1, z_2) \; \Leftrightarrow \; |z_2 + 1| > |z_2 - 1| \; \Leftrightarrow \mathrm{Re}(z_2) > 0.
		\]
	\end{proof}

	Finally, we are ready to give sufficient conditions on a triple of points $z_1,z_2,z_3 \in \CC$ for them to satisfy $z_3 \notin F(z_1, z_2)$.
	
	\begin{lemma}\label{lemInSameStrip}
		Let $z_1, z_2, z_3 \in \Sigma$ satisfy the following conditions.
		\begin{enumerate}[label=(\roman*)]
			\item $|z_3 - z_1|, |z_3 - z_2| \ge 2$.
			\item $\mathrm{Re}(z_1) \le \mathrm{Re}(z_2) \le \mathrm{Re}(z_3)$ and, if $\mathrm{Re}(z_2) = \mathrm{Re}(z_3)$, then $\mathrm{Im}(z_3) \ne \mathrm{Im}(z_1)$.
		\end{enumerate}
		Then $z_3 \notin F(z_1, z_2)$.
	\end{lemma}
	\begin{proof}
		
		Let $z_j = \alpha_j + i\beta_j$, where $\alpha_j, \beta_j \in \RR$ for $j \in \{1,2,3\}$. Due to Lemma~\ref{lemCriterionForDiscOOnDiameter} we need to prove that 
		\[
		\mathrm{Re}\left((z_3 - z_1)(\overline{z_3} - \overline{z_2})\right)  = (\alpha_3 - \alpha_1)(\alpha_3 - \alpha_2) + (\beta_3 - \beta_1)(\beta_3 - \beta_2) > 0.
		\]
		By assumptions $\alpha_3 \ge \max\{\alpha_1, \alpha_2\}$. Therefore, if $(\beta_3 - \beta_1)(\beta_3 - \beta_2) > 0$, then the statement, obviously, holds. Moreover, if $(\beta_3 - \beta_1)(\beta_3 - \beta_2) = 0$, then the statement is true, unless $(\alpha_3 - \alpha_1)(\alpha_3 - \alpha_2) = 0$. But this may happen only if $\alpha_3 = \alpha_2$ and, by assumptions, in this case $\beta_3 \ne \beta_1$. Moreover, $z_3 \ne z_2$, so $\beta_3 \ne \beta_2$. Therefore, $(\beta_3 - \beta_1)(\beta_3 - \beta_2) \ne 0$, contradicting the assumption. 
		
		It remains to consider the case, when $(\beta_3 - \beta_1)(\beta_3 - \beta_2) < 0$. From the inequalities $|z_3 - z_1| \ge 2$ and $|z_3 - z_2| \ge 2$ we conclude that
		\[
		(\alpha_3 - \alpha_1)^2 + (\beta_3 - \beta_1)^2 \ge 4,\;\;(\alpha_3 - \alpha_2)^2 + (\beta_3 - \beta_2)^2 \ge 4.
		\]
		From these inequalities it follows that
		\begin{equation*}
			(\alpha_3 - \alpha_1)^2(\alpha_3 - \alpha_2)^2 \ge (4 - (\beta_3 - \beta_1)^2)(4 - (\beta_3 - \beta_2)^2).
		\end{equation*}
		Let $s = |\beta_3 - \beta_1|$ and $t = |\beta_2 - \beta_1|$. From the inequality $(\beta_3 - \beta_1)(\beta_3 - \beta_2) < 0$ it follows that $s, t > 0$ and $s + t = |\beta_1 - \beta_2| \le 2$. Elementary calculations show that for any $s,t > 0$ such that $s + t \le 2$ we have the inequality $(4 - t^2)(4 - s^2) > t^2s^2$. Therefore,
		\[
		(\alpha_3 - \alpha_1)^2(\alpha_3 - \alpha_2)^2 \ge (4 - t^2)(4 - s^2) > t^2s^2 = (\beta_3 - \beta_1)^2(\beta_3 - \beta_2)^2.
		\]
		Since $(\alpha_3 - \alpha_1)(\alpha_3 - \alpha_2) \ge 0$, this implies the required inequality.
	\end{proof}
	
	\begin{lemma}\label{lemLargestNotInTheSameStrip}
		Let $z_1, z_2, z_3 \in \CC$ satisfy the following conditions.
		\begin{enumerate}[label=(\roman*)]
			\item\label{LargestNoti}  $|z_3| \ge \max\{|z_1|, |z_2|\}$.
			\item\label{LargestNotii} $z_1$ and $z_2$ are approximately collinear, and $\min\{|z_1|, |z_2|\} > 1$.
			\item\label{LargestNotiii} There is no $a \in \TT$ such that $|\mathrm{Im}(z_j/a)| \le 1$ for all $j \in \{1,2,3\}$.
		\end{enumerate}
		Then $z_3 \notin F(z_1, z_2)$.
	\end{lemma}
	\begin{proof}
		Clearly, by~\ref{LargestNotii} Lemma~\ref{lemEqualImagPartsInStrip} is applicable to $z_1$ and $z_2$. Therefore, we can replace $z_j$ with $z_j/b$, where $b$ is given by Lemma~\ref{lemEqualImagPartsInStrip} applied to $z_1$ and $z_2$. Thus, from now on we assume that $z_1,z_2 \in \Sigma$ and $\mathrm{Im}(z_1) = -\mathrm{Im}(z_2)$. The statement~\ref{LargestNotiii} implies that $|\mathrm{Im}(z_3)| > 1$. Clearly, due to~\ref{LargestNoti}, it suffices to prove that if $w \in F(z_1,z_2)$ and $|\mathrm{Im}(w)| > 1$, then $|w| < \max\{|z_1|, |z_2|\}$. Without loss of generality we may assume that $\mathrm{Re}(z_1) \le \mathrm{Re}(z_2)$, so we may write $z_1 = \alpha+ i\delta$ and $z_2 = \beta - i\delta$, where $0 \le \alpha \le \beta$ and $|\delta| < 1$. Now consider arbitrary $w = a + ib \in F(z_1, z_2)$ such that $|b| > 1$. As $w \in F(z_1,z_2)$, from Lemma~\ref{lemCriterionForDiscOOnDiameter} it follows that 
		\[
		(a - \alpha)(a - \beta) + b^2 - \delta^2 \le 0.
		\]
		Since $|b| > 1$ and $|\delta| \le 1$ it follows that $(a - \alpha)(a - \beta) < 0$. Thus, $\alpha < a < \beta$ (in particular, $\beta > 0$). Therefore, we can write
		\[
		|w|^2 = a^2 + b^2 \le \delta^2 + \alpha a + \beta a - \alpha \beta < \delta^2 + \alpha \beta + \beta^2 - \alpha \beta = \delta^2 + \beta^2 = |z_2|^2.
		\]
	\end{proof}
	
	Now we can easily prove the following lemma.
	
	\begin{lemma}\label{lemIfSharpAndImpersandPartialThenSplittable}
		Assume that a three-point set $E \subset \CC$ is distinguished and well-separated from zero. Then the following statements hold.
		\begin{enumerate}[label=(\roman*)]
			\item\label{partialcasesi} If $E$ satisfies condition~\ref{casei} of Proposition~\ref{propIfSharpNotAsteriskThenCases}, then $\{-1,1,\infty\} \cup E$ is splittable.
			\item\label{partialcasesii} Assume that $E$ satisfies condition~\ref{caseii} of Proposition~\ref{propIfSharpNotAsteriskThenCases} and let $E = \{z_1, z_2,z_3\}$ be the corresponding enumeration. If $|z_1| \ge |z_2|$, then $\{-1,1,\infty\} \cup E$ is splittable.
		\end{enumerate}
	\end{lemma}
	\begin{proof}
		To prove~\ref{partialcasesi} assume that $E$ satisfies condition~\ref{casei} of Proposition~\ref{propIfSharpNotAsteriskThenCases}. Then for some $a \in \TT$ such that for all $z \in E$ we have inequalities $\mathrm{Re}(z/a) \ge 0$ and $|\mathrm{Im}(z/a)| \le 1$. Clearly, we can choose an ordering $E = \{z_1, z_2, z_3\}$ such that $\mathrm{Re}(z_1/a) \le \mathrm{Re}(z_2/a) \le \mathrm{Re}(z_3/a)$ and if $\mathrm{Re}(z_2/a) = \mathrm{Re}(z_3/a)$, then $\mathrm{Im}(z_3/a) \ne \mathrm{Im}(z_1/a)$. Let $F = F(z_1/a, z_2/a)$. By Lemma~\ref{lemInSameStrip} $z_3/a \notin F$, and, since $F$ is well-separated from zero, we conclude that $F \cap \overline{\mathbb D} = \emptyset$. Moreover, by Lemma~\ref{lemComparisonInStrip} we have $|z_3/a| \ge \max\{|z_1/a|, |z_2/a|\}$, so $z_3/a \notin \conv(F \cup \overline{\mathbb D})$ by Lemma~\ref{lemBiggestDoesNotBelongToConvexHull}. Since, obviously, $1/a$ and $-1/a$ belong to $\overline{\mathbb D}$, by Lemma~\ref{lemFinishingSplitting}, the set
		\[
		A = \left\{\dfrac{1}{a}, -\dfrac{1}{a}, \dfrac{z_1}{a}, \dfrac{z_2}{a}, \dfrac{z_3}{a}, \infty\right\}
		\]
		is splittable. Therefore, $\{-1, 1, \infty\} \cup E$ is splittable, for $A$ is the image of this set with respect to scalar multiplication by $a^{-1}$.
		
		Now we prove~\ref{partialcasesii}. From Proposition~\ref{propIfSharpNotAsteriskThenCases}~\ref{caseii} and Lemma~\ref{lemIfLargeImaginaryOppositeSignsThenNotApprox} it is easy to conclude that  $z_1$ and $z_3$ are not approximately collinear. Thus, we may only consider the case, when there is no $a \in \TT$ such that $|\mathrm{Im}(z_j/a)| \le 1$ for all $j \in \{1,2,3\}$ (indeed, if such $a$ exists, then, since $z_1$ and $z_3$ are not approximately collinear, $\mathrm{Re}(z_1/a)$ and $\mathrm{Re}(z_3/a)$ have the opposite signs; but then $E$ is splittable by a strip due to Lemma~\ref{lemIfNotAsteriskThenNoOppositeStrips} and the proof is finished by applying Proposition~\ref{propSplittingIfAsterisk}). From the assumption $|z_1| \ge |z_2|$ we can conclude that either $|z_1| \ge \max \{|z_2|, |z_3|\}$, or $|z_3| \ge \max \{|z_1|, |z_2|\}$. Both pairs $(z_1, z_2)$ and $(z_2, z_3)$ are approximately collinear (the first pair due to conditions~\ref{I} and~\ref{II}, and the second due to~\ref{IV}). Thus, we may enumerate $E$ in a way $E = \{w_1, w_2, w_3\}$ such that $w_1$ and $w_2$ are approximately collinear, and $ |w_3| \ge \max\{|w_1|, |w_2|\}$. By Lemma~\ref{lemLargestNotInTheSameStrip} we see that $w_3 \notin F$, where $F = F(w_1, w_2)$. Further, since $ |w_3| \ge \max\{|w_1|, |w_2|\}$ by Lemma~\ref{lemBiggestDoesNotBelongToConvexHull} we conclude that $w_3 \notin \conv(F \cup \overline{\mathbb D})$. Finally, $F \cap \overline{\mathbb D} = \emptyset$, since $F$ is well-separated from zero, and $w_1$ and $w_2$ are approximately collinear. Thus, $\{-1,1,\infty\} \cup E$ is splittable by Lemma~\ref{lemFinishingSplitting}.
	\end{proof}
	
	The last case of Proposition~\ref{propIfSharpNotAsteriskThenCases} that is not covered in Lemma~\ref{lemIfSharpAndImpersandPartialThenSplittable} can be handled using the following fact (see Fig.~\ref{figDeformedCircleInStrip}).
	\FloatBarrier
	\begin{lemma}\label{lemDeformedCircleInImpersandCase}
		Let $z_1, z_2 \in \CC$ satisfy $1 \le \mathrm{Re}(z_1) < \mathrm{Re}(z_2)$, $\mathrm{Im}(z_1) = 1$, and $|\mathrm{Im}(z_2)| \le 1$. Moreover, assume that $F(z_1, z_2) \cap \overline{\mathbb D} = \emptyset$. Then there exists a closed disc $F \subset \CC$ such that $z_1, z_2 \in F$, $F \cap \overline{\mathbb D} = \emptyset$, and $F \subset \{z \in \CC: \mathrm{Im}(z) \ge -1\}$.
	\end{lemma}

	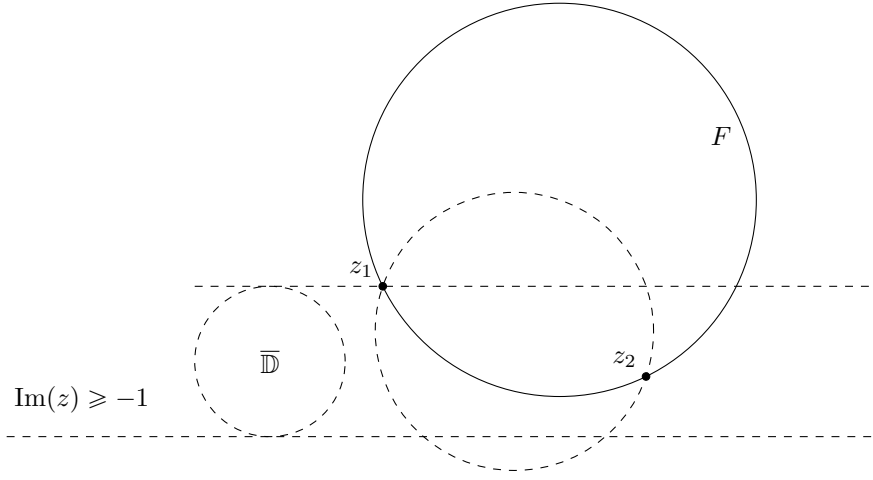
\begin{figure}
		\begin{tikzpicture}
			\draw[dashed] (0,0) circle[radius = 1];
			\draw[dashed] (-3.5, -1) -- (8, -1);
			\draw[dashed] (-1, 1) -- (8, 1);
			\draw[dashed] (3.25, 0.4) circle[radius = {sqrt(49/16 + 9/25)}];
			\filldraw (1.5, 1) circle[radius = 0.05] node[above left] {$z_1$};
			\filldraw (5, -0.2) circle[radius = 0.05] node[above left] {$z_2$};
			\draw ({3.25 + 1.2/2}, {0.4 + 3.5/2}) circle[radius = {sqrt(49/8 + 18/25)}];
			\node at (6, 3) {$F$};
			\node at (0,0) {$\overline{\DD}$};
			\node at (-2.5, -0.5) {$\mathrm{Im}(z) \ge -1$};
		\end{tikzpicture}
		\caption{The disc $F$ from Lemma~\ref{lemDeformedCircleInImpersandCase}.}
		\label{figDeformedCircleInStrip}
	\end{figure}
	\begin{proof}
		Let $c = (z_1 + z_2)/2$ and $v = i(z_2 - z_1)/2$. For $t \ge 0$ we consider the closed disc $F(t) = \{c + tv + |v|\sqrt{1 + t^2}\cdot\overline{\mathbb D}\}$. From the calculation
		\[
		|c + tv - z_1| = \left| \dfrac{z_2 - z_1}{2} + \dfrac{it(z_2 - z_1)}{2} \right| = |v|\sqrt{1 + t^2} = |c + tv - z_2|
		\]
		it is clear that $z_1, z_2 \in F(t)$ for all $t$.
		
		At first we check that $F(t) \cap \overline{\mathbb D} = \emptyset$ for all $t \ge 0$. Note that $F(0) = F(z_1,z_2)$, so $F(0) \cap \overline{\mathbb D} = \emptyset$ by assumtions. Consider the half-plane $H = \{c + s v: s \in \CC, \mathrm{Re}(s) \le 0\}$. We claim that $\overline{\mathbb D} \subset H$. Assume that $z \notin H$, so $z = c + i\alpha v + \beta v$, where $\alpha, \beta \in \mathbb R$, $\beta > 0$. Since $z_1 = c + iv$, we can also write $z = z_1 + i(\alpha - 1)v + \beta v$. Now note, that from assumptions on $z_1$ and $z_2$ it follows that $\mathrm{Re}(v) \ge 0$ and $\mathrm{Im}(v) > 0$. If $\alpha - 1 \ge 0$, then $\mathrm{Im}(i (\alpha - 1) v) \ge 0$, and therefore, $\mathrm{Im}(z) > \mathrm{Re}(z_1) = 1$. On the other hand, if $\alpha - 1 < 0$, then $\mathrm{Re}(i(\alpha - 1)v) > 0$, so $\mathrm{Re}(z) > \mathrm{Im}(z_1) \ge 1$. Thus, in any case $|z| > 1$. It follows that $\CC \setminus H$ does not intersect $\overline{\mathbb D}$, so $\overline{\mathbb D} \subset H$. Now we claim that $F(t) \cap H \subset F(0) \cap H$ for all $t \ge 0$. Indeed, assume that $z \in F(t) \cap H$. Then $z = c + i \alpha v + \beta v$, where $\alpha, \beta \in \RR$, $\beta \le 0$. Moreover, since $z \in F(t)$, we have the inequality $|z - c - tv| \le |v|\sqrt{1 + t^2}$. This means that
		\[
		|i\alpha v + \beta v - tv| \le |v| \sqrt{1 + t^2} \Rightarrow (\beta - t)^2 + \alpha^2 \le 1 + t^2 \Rightarrow \beta^2 + \alpha^2 - 2\beta t \le 1.
		\]
		Since $\beta \le 0$ and $t \ge 0$ we have $\alpha^2 + \beta^2 \le 1 + 2 \beta t \le 1$. Thus, $|z - c| \le |v|$, i.e. $z \in F(0)$. Finally, we can conclude that $F(t) \cap \overline{\mathbb D} = \emptyset$ for all $t \ge 0$. Indeed, \[
		F(t) \cap \overline{\mathbb D} \subset F(t) \cap \overline{\mathbb D} \cap H \subset F(0) \cap \overline{\mathbb D} \cap H \subset F(0) \cap \overline{\mathbb D} = \emptyset.
		\]
		
		To finish the proof it remains to show that for some $t \ge 0$ we have the inclusion $F(t) \subset \{z \in \CC: \mathrm{Im}(z) \ge -1\}$. At first we consider the case when $\mathrm{Im}(z_2) < 1$. Then $\mathrm{Re}(v) > 0$. We claim that we may choose $t \ge 0$ in order to satisfy the equality $\mathrm{Re}(c + tv) = \mathrm{Re}(z_2)$. Indeed, since $z_2 = c - iv$, this is equivalent to
		\[
		t\mathrm{Re}(v) = -\mathrm{Re}(iv) = \mathrm{Im}(v) \Rightarrow t = \dfrac{\mathrm{Im}(v)}{\mathrm{Re}(v)} \ge 0.
		\]
		We already calculated that $\sqrt{1 + t^2}|v| = |c + tv - z_2|$. Since $\mathrm{Re}(c + tv) = \mathrm{Re}(z_2)$ and $\mathrm{Im}(c + tv) \ge \mathrm{Im}(z_2)$, it follows that $c + tv = z_2 + i\sqrt{1 + t^2} |v|$. Thus, for all $z \in F(t)$ we have $\mathrm{Im}(z) \ge \mathrm{Im}(c + tv) - \sqrt{1 + t^2} |v| = \mathrm{Im}(z_2) \ge -1$. It remains to consider the case $\mathrm{Im}(z_2) = 1$. But then $\mathrm{Im}(c) = 1$ and $\mathrm{Re}(v) = 0$. Thus, to have the inclusion $F(t) \subset \{z \in \CC: \mathrm{Im}(z) \ge -1\}$ it suffices to satisfy the inequality $1 + t \mathrm{Im}(v) - |v|\sqrt{1 + t^2} \ge -1$. But this inequality is clearly satisfied for large $t$, as $\mathrm{Im}(v) = |v|$, and $t - \sqrt{1 + t^2} \to 0$ as $t \to +\infty$.
	\end{proof}
	
	\begin{proof}[Proof of Proposition~\ref{propIfWellSeparatedThenSplittable}]
		If $E$ is splittable by a strip, then $\{-1, 1, \infty\} \cup E$ is splittable due to Proposition~\ref{propReductionToSharp}. Thus, we may assume that $E$ is not splittable by a strip. Since $E$ is also distinguished, one of the cases of Proposition~\ref{propIfSharpNotAsteriskThenCases} holds for $E$. If case~\ref{casei} holds, then $\{-1, 1, \infty\} \cup E$ is splittable by Lemma~\ref{lemIfSharpAndImpersandPartialThenSplittable}~\ref{partialcasesi}. Now we assume that $E$ satisfies Proposition~\ref{propIfSharpNotAsteriskThenCases}~\ref{caseii}. Let the number $a \in \TT$ and the enumeration $E = \{z_1, z_2, z_3\}$ satisfy the conditions of Proposition~\ref{propIfSharpNotAsteriskThenCases}~\ref{caseii}. If $|z_1| \ge |z_2|$, then $\{-1, 1, \infty\} \cup E$ is splittable by Lemma~\ref{lemIfSharpAndImpersandPartialThenSplittable}~\ref{partialcasesii}. Thus, we may assume that $|z_1| < |z_2|$. By Lemma~\ref{lemComparisonInStrip} this means that $\mathrm{Re}(z_1/a) < \mathrm{Re}(z_2/a)$. On the other hand, from Proposition~\ref{propIfSharpNotAsteriskThenCases}~\ref{caseii}~\ref{I} and~\ref{II} we have $\mathrm{Im}(z_1/a) = 1$ and $|\mathrm{Im}(z_2/a)| \le 1$. Moreover, as $z_1 \notin \Omega$ we have $|z_1| \ge \sqrt{3}$, so $\mathrm{Re}(z_1/a) \ge \sqrt{2} > 1$. Thus, the points $z_1/a$ and $z_2/a$ satisfy the conditions of Lemma~\ref{lemDeformedCircleInImpersandCase} (we have $F(z_1/a, z_2/a) \cap \overline{\mathbb D} = \emptyset$, since $E$ is well-separated from zero, and $z_1$ and $z_2$ are approximately collinear). Thus, there exists a closed disc $F \subset \CC$ such that $z_1/a, z_2/a \in F$, $F \cap \overline{\mathbb D} = \emptyset$, and $F  \subset \{z \in \CC: \mathrm{Im}(z) \ge -1\}$. Since $\mathrm{Im}(z_3/a) < -1$ by Proposition~\ref{propIfSharpNotAsteriskThenCases}~\ref{caseii}~\ref{III}, we conclude that $z_3/a \notin \conv(F \cup \overline{\mathbb D})$. Thus, the set 
		\[
		A = \left\{\dfrac{1}{a}, -\dfrac{1}{a}, \dfrac{z_1}{a}, \dfrac{z_2}{a}, \dfrac{z_3}{a}, \infty\right\}
		\]
		is splittable, and so is $\{-1, 1, \infty\} \cup E$.
	\end{proof}
	
	\section{Splitting in the remaining cases}\label{sec:Calculations}
	\FloatBarrier
	Now it remains only to prove splittability of $\{-1,1,\infty\} \cup E$, if $E$ is distinguished and not well-separated from zero, i.e. it contains a pair of approximately collinear points $z_1,z_2$ such that $F(z_1,z_2) \cap \overline{\mathbb \overline{D}} \ne \emptyset$. This situation is possible (see Fig.~\ref{figExampleOfIntersectingDiscs}), but we shall see that the positions of $z_1$ and $z_2$ are very restricted. We start by collecting some properties of $z_1$ and $z_2$.
	
	\begin{figure}
		\begin{tikzpicture}
			\draw[dashed] (-0.25, -2) -- (3.5, 3);
			\draw[dashed] (-2.75, -2) -- (1, 3);
			\draw (0,0) circle[radius = 1];
			\draw[black, dashed, domain = -120:120] plot({1 + 2*cos(\x)}, {2*sin(\x)});
			\draw[black, dashed, domain = 60:300] plot({-1 + 2*cos(\x)}, {2*sin(\x)});
			\filldraw (0.13, 1.83) circle[radius = 0.05] node[above left] {$z_1$};
			\filldraw (2.33, 1.55) circle[radius = 0.05] node[above left, yshift=3] {$z_2$};
			\draw (1.23, 1.69) circle[radius = {sqrt(2.2*2.2 + 0.28*0.28)/2}];
		\end{tikzpicture}
		\caption{Example of a pair of approximately collinear points $z_1, z_2 \in \CC \setminus \Omega$ such that $F(z_1,z_2) \cap \overline{\DD} \ne \emptyset$.}
		\label{figExampleOfIntersectingDiscs}
	\end{figure}
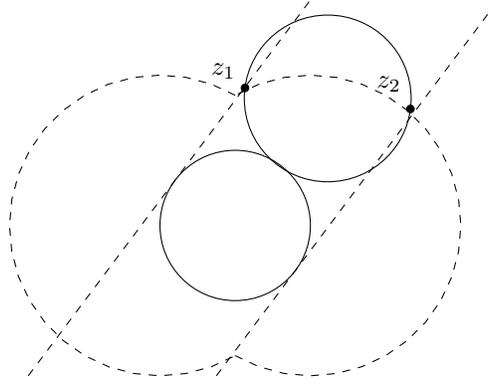

	\begin{lemma}\label{lemPropertiesOfIntersectingCirclesBasic}
		Let $z_1,z_2 \in \CC \setminus \Omega$ be a pair of approximately collinear points such that $F(z_1,z_2) \cap \overline{\mathbb D} \ne \emptyset$. Assume that $|z_1| \le |z_2|$. Then the following statements hold.
		\begin{enumerate}[label=(\roman*)]
			\item\label{CirclePropertyi} $|z_1| \le \sqrt{5}$.
			\item\label{CirclePropertyii} $|z_1 - z_2| \le 2\sqrt{2}$.
			\item\label{CirclePropertyiii} If $|z_1 - z_2| = 2\sqrt{2}$, then $z_1 = \pm \sqrt{3}i$.
			\item\label{CirclePropertyiv} Either $\mathrm{Im}(z_1), \mathrm{Im}(z_2) > 0$, or $\mathrm{Im}(z_1), \mathrm{Im}(z_2) < 0$.
		\end{enumerate}
	\end{lemma}
	\begin{proof}
		First of all we choose (using Lemma~\ref{lemEqualImagPartsInStrip}) $a \in \TT$ such that $z_1/a, z_2/a \in \Sigma$ and $\mathrm{Im}(z_1/a) = -\mathrm{Im}(z_2/a)$. Then $z_1/a = \alpha + i\delta$ and $z_2/a = \beta - i\delta$, where $|\delta| \le 1$ and $0 \le \alpha \le \beta$ by Lemma~\ref{lemComparisonInStrip}. In the proof we shall several times use the following fact: if $z \in F(z_1/a, z_2/a)$, then
		\begin{equation}\label{eqSmallestRealPartInCircle}
			\mathrm{Re}(z) \ge g(\alpha, \beta, \delta) = \dfrac{\alpha+ \beta}{2} - \sqrt{\dfrac{(\alpha - \beta)^2}{4} + \delta^2}.
		\end{equation}
		By differentiation it is easy to verify that $g(\alpha, \beta, \delta)$ is nondecreasing with respect to $\alpha$ and $\beta$, and decreasing with respect to $|\delta|$ (when other variables are fixed).
		
		To prove~\ref{CirclePropertyi} assume the contrary, i.e. $|z_1| > \sqrt{5}$. Thus, $\beta \ge \alpha > 2$. Now for $z \in F(z_1/a, z_2/a)$ (applying~\eqref{eqSmallestRealPartInCircle} and an elementary inequality $\sqrt{t^2 + 1} \le t + 1$, which holds for all $t \ge 0$) we estimate
		\begin{multline*}
			\mathrm{Re}(z) \ge \dfrac{\alpha+ \beta}{2} - \sqrt{\dfrac{(\alpha - \beta)^2}{4} + \delta^2} \ge \dfrac{\alpha+ \beta}{2} - \sqrt{\dfrac{(\alpha - \beta)^2}{4} + 1} \ge \\ \dfrac{\alpha+ \beta}{2} - \dfrac{\beta - \alpha}{2} - 1 = \alpha - 1 > 1.
		\end{multline*}
		Thus, $\mathrm{Re}(z) > 1$ for all $z \in F(z_1/a, z_2/a)$, hence, $F(z_1/a, z_2/a) \cap \overline{\mathbb D} = \emptyset$. Clearly, this yields $F(z_1, z_2) \cap \overline{\mathbb D} = \emptyset$, leading to a contradiction.
		
		Now we prove~\ref{CirclePropertyii} and~\ref{CirclePropertyiii}. It is clear that $\alpha \ge \sqrt{2}$, for $|z_1| \ge \sqrt{3}$. It is easy to calculate that $g(\sqrt{2}, 2 + \sqrt{2}, 1) = 1$. Thus, by monotonicity properties of $g$ it is straightforward to prove that if $\beta - \alpha > 2$, $\alpha \ge \sqrt{2}$, and $|\delta| \le 1$, then $g(\alpha, \beta, \delta) > 1$. Since $F(z_1, z_2) \cap \overline{\mathbb D} \ne \emptyset$, then $g(\alpha, \beta, \delta) \le 1$, so we can conclude that $\beta - \alpha \le 2$. Thus, $|z_1 - z_2|^2 \le (\beta - \alpha)^2 + (2\delta)^2 \le 8$. That is, $|z_1 - z_2| \le 2\sqrt{2}$, and~\ref{CirclePropertyii} is proved. Moreover, it is easy to see that the equality $|z_1 - z_2| = 2\sqrt{2}$ is possible only if $\beta - \alpha = 2$ and $|\delta| = 1$. But the inequalities $g(\alpha, \beta, \delta) \le 1$ and $\alpha \ge \sqrt{2}$ then imply $\alpha = \sqrt{2}$ (otherwise, monotonicity of $g$ would imply that $g(\alpha, \beta, \delta) > 1$). Thus, $|z_1|^2 = \alpha^2 + 1 = 3$. To prove~\ref{CirclePropertyiii} it now suffices to note that the only two points in $\CC \setminus \Omega$ with absolute value $\sqrt{3}$ are $\pm\sqrt{3}i$.
		
		It remains to prove~\ref{CirclePropertyiv}. Clearly we may assume without loss of generality that $\mathrm{Im}(z_1) \ge 0$. Since $|z_1| \le \sqrt{5}$ and $|z_1 - 1|, |z_1 + 1| \ge 2$ it is easy to see that $|\mathrm{Re}(z_1)| \le 1$. Moreover, for all $z \in \CC \setminus \Omega$ such that $|\mathrm{Re}(z)| \le 1$ the inequality $|\mathrm{Im}(z)| \ge \sqrt{3}$ holds. Thus, $\mathrm{Im}(z_1) \ge \sqrt{3}$. Now we prove that $\mathrm{Im}(z_2) > 0$. Clearly, if $\mathrm{Im}(z_2) \le -1$, then $z_1$ and $z_2$ are not approximately collinear by Lemma~\ref{lemIfLargeImaginaryOppositeSignsThenNotApprox}. Thus, it remains to consider $z_2 \in \CC \setminus \Omega$ such that $\mathrm{Im}(z_2) \in [-1, 0]$. Assume, in addition, that $\mathrm{Re}(z_2) \ge 0$ (the other case is handled similarly). Then it is easy to see that $\mathrm{Re}(z_2) \ge 1 + \sqrt{3}$ (otherwise, we would have $|z_2 - 1| < 2|$). Then consider $a \in \TT$ such that $\mathrm{Im}(a) = 1/(1 + \sqrt{3})$ and $\mathrm{Re}(a) > 0$. It is easy to see that
		\[
		\mathrm{Im}(z_2/a) = \mathrm{Im}(z_2)\mathrm{Re}(a) - \mathrm{Re}(z_2)\mathrm{Im}(a) \le -1.
		\]
		Moreover, we can calculate that 
		\[
		a = \dfrac{\sqrt{\sqrt{3}(2+\sqrt{3})}}{1 + \sqrt{3}} + \dfrac{i}{1 + \sqrt{3}},
		\]
		so by using the already established inequalities $\mathrm{Im}(z_1) \ge \sqrt{3}$ and $\mathrm{Re}(z_1) \ge -1$ we obtain that
		\[
		\mathrm{Im}(z_1/a) = \mathrm{Im}(z_1)\mathrm{Re}(a) - \mathrm{Re}(z_1)\mathrm{Im}(a) \ge \dfrac{\sqrt{3\sqrt{3}(2+\sqrt{3})} - 1}{1 + \sqrt{3}}.
		\]
		An easy calculation shows that the last estimate on $\mathrm{Im}(z_1/a)$ is greater than $1$. Thus, $z_1/a$ and $z_2/a$ are not approximately collinear by Lemma~\ref{lemIfLargeImaginaryOppositeSignsThenNotApprox}.
	\end{proof}
	
	To prove some more involved statements than those listed in Lemma~\ref{lemPropertiesOfIntersectingCirclesBasic} we shall use parametric representation of points $z \in \CC \setminus \Omega$ with small absolute value.
	\begin{lemma}\label{lemAnglesOfSmallAbsoluteValuesOutsideOmega}
		Let $z \in \CC \setminus \Omega$ satisfy $|z| \le 3$ and $\mathrm{Im}(z) \ge 0$. Then $z$ can be represented as $z = re^{i\alpha}$, where $r = |z|$ and $\alpha \in [q(r), \pi - q(r)]$, $q(r) = \arccos((r^2 - 3)/2r)$.
	\end{lemma}
	\begin{proof}
		Clearly, $z$ can be represented as $z = re^{i\alpha}$, where $\alpha \in [0, \pi]$.
		Then, by a straightforward calculation, one can show that $|z - 1| < 2$ if $z = re^{i\alpha}$ where $\alpha \in [0,\pi]$ with $\cos(\alpha) > (r^2 - 3)/2r$ (similarly, $|z + 1| < 2$ if $\cos(\alpha) < -(r^2 - 3)/2r)$.
	\end{proof}
	
	The following lemma states that if $z_1, z_2 \in \CC\setminus \Omega$ are approximately collinear, $|z_1-z_2| \ge 2$, and $F(z_1,z_2) \cap \overline{\mathbb D} \ne \emptyset$, then $F(z_1, z_2)$ still cannot be too close to zero (it is not separate from zero by distance $1$ anymore, but we can obtain a lower estimate on $|z|$ for $z \in F(z_1,z_2)$).
	
	\begin{lemma}\label{lemPropertiesOfIntersectingCirclesInvolved}
		Let $z_1,z_2 \in \CC \setminus \Omega$ be a pair of approximately collinear points such that $F(z_1,z_2) \cap \overline{\mathbb D} \ne \emptyset$. Assume that $|z_1| \le |z_2|$ and $|z_1 - z_2| \ge 2$. Then the following statements hold.
		\begin{enumerate}[label=(\roman*)]
			\item\label{CirclePropertyIi} $|z_2| \ge \sqrt{10 - |z_1|^2}$.
			\item\label{CirclePropertyIii} For all $z \in F(z_1,z_2)$ we have the inequality
			\[|z| \ge \dfrac{1}{2}\left(\sqrt{|z_1|^2 - 1} + \sqrt{9 - |z_1|^2} - \sqrt{12 -2\sqrt{(|z_1|^2 - 1)(9 - |z_1|^2)}}\right).\]
		\end{enumerate}
	\end{lemma}
	\begin{proof}	
		We start by proving~\ref{CirclePropertyIi}. Due to~Lemma~\ref{lemPropertiesOfIntersectingCirclesBasic}~\ref{CirclePropertyiv} we may without loss of generality assume that $\mathrm{Im}(z_1), \mathrm{Im}(z_2) > 0$. Assume that $r = |z_1|$ and $\rho = |z_2| < \sqrt{10 - |z_1|^2}$. By Lemma~\ref{lemAnglesOfSmallAbsoluteValuesOutsideOmega} we may write $z_1 = re^{i\alpha}$ and $z_2 = \rho e^{i\beta}$, where $\alpha \in [q(r), \pi - q(r)]$ and $\beta \in [q(\rho), \pi - q(\rho)]$ (the value $q(r)$ is defined in Lemma~\ref{lemAnglesOfSmallAbsoluteValuesOutsideOmega}). Therefore, 
		\[
		|z_1 - z_2|^2 = (r e^{i\alpha} - \rho e^{i\beta})(r e^{-i\alpha} - \rho e^{-i\beta}) = r^2 + \rho^2 - 2\rho r \cos(\alpha - \beta).
		\]
		Clearly, for $\alpha \in [q(r), \pi - q(r)]$ and $\beta \in [q(\rho), \pi - q(\rho)]$ we can estimate $\cos(\alpha - \beta) \ge \cos(q(r) + q(\rho))$. Further, by the definition of $q(r)$ we have
		\[
		\cos(q(r)) = \dfrac{r^2 - 3}{2r}, \;\;\sin(q(r)) = \dfrac{\sqrt{10r^2 - r^4 - 9}}{2r} = \dfrac{\sqrt{(r^2 -1)(9 - r^2)}}{2r}.
		\]
		Therefore, 
		\[
		\cos(q(r) + q(\rho)) = \dfrac{(r^2 - 3)(\rho^2 - 3)}{4r\rho} - \dfrac{\sqrt{(r^2 -1)(9 - r^2)}\sqrt{(\rho^2 -1)(9 - \rho^2)}}{4r\rho}.
		\]
		Thus, we obtained the inequality $|z_1 - z_2|^2 \le f(\rho^2)$, where
		\[
		f(t) = t + r^2 - \dfrac{1}{2}(r^2 - 3)(t - 3) + \dfrac{1}{2}\sqrt{(r^2 -1)(9 - r^2)}\sqrt{(t -1)(9 - t)}
		\]
		It is straightforward to verify that $f(10 - r^2) = 4$. Thus, if $f$ strictly increases on the interval $(r^2, 10 - r^2)$, then $f(\rho^2) < 4$. This contradicts the assumption $|z_1 - z_2| \ge 2$. Therefore, to prove~\ref{CirclePropertyIi} it remains to verify that $f'(t) > 0$ for all $t \in (r^2, 10 - r^2)$. By an elementary calculation we have
		\[
		f'(t) = 1 - \dfrac{1}{2}(r^2 - 3) + \dfrac{(5 - t)\sqrt{(r^2 -1)(9 - r^2)}}{2\sqrt{(t -1)(9 - t)}}.
		\]
		By substituting there $t = 10 - r^2$ we obtain that $f'(10 - r^2) = 0$. By differentiating again we obtain that
		\[
		f''(t) = -\left({\sqrt{(r^2 -1)(9 - r^2)}}{2\sqrt{(t -1)(9 - t)}} + \dfrac{(5 - t)^2\sqrt{(r^2 -1)(9 - r^2)}}{4\sqrt{(t -1)^3(9 - t)^3}}\right) < 0
		\]
		for $t \in (1, 9)$. In particular, $f''(t) < 0$ for $t \in (r^2, 10 - r^2) \subset (3, 7)$ (note that $|z_1| \le \sqrt{5}$ by Lemma~\ref{lemPropertiesOfIntersectingCirclesBasic}~\ref{CirclePropertyi}), so $f'(t) > 0$ for all $t \in (r^2, 10 -r^2)$.
		
		To prove~\ref{CirclePropertyIii} choose $a \in \CC$ such that $|a| = 1$, $z_1/a$ and $z_2/a$ belong to $\Sigma$, and $\mathrm{Im}(z_1/a) = -\mathrm{Im}(z_2/a)$ (see Lemma~\ref{lemEqualImagPartsInStrip}). Thus, $z_1/a = \alpha + i\delta$ and $z_2/a = \beta - i\delta$, where $\beta \ge \alpha \ge 0$ and $\delta \in [-1,1]$ (the inequality $\alpha \le \beta$ follows from Lemma~\ref{lemComparisonInStrip}, as $|z_1| \le |z_2|$). Then (see~\eqref{eqSmallestRealPartInCircle})
		\begin{multline*}
			\inf \{|z|: z\in F(z_1,z_2)\} = \inf \{|z|: z\in F(z_1/a, z_2/a)\} = \\ \inf \left\{|z|: z \in \dfrac{\alpha + \beta}{2} + \left(\sqrt{\dfrac{(\alpha - \beta)^2}{4} + \delta^2}\right)\cdot \overline{\mathbb D} \right\} \ge \\ g(\alpha, \beta, \delta) = \dfrac{\alpha + \beta}{2} - \sqrt{\dfrac{(\alpha - \beta)^2}{4} + \delta^2}.
		\end{multline*}
		By~\ref{CirclePropertyIi} $\beta^2 + \delta^2 \ge 10 - \alpha^2 - \delta^2$, so $\beta \ge \sqrt{10 - \alpha^2 - 2\delta^2}$. Therefore, since $g(\alpha, \beta, \delta)$ is nondecreasing with respect to $\beta$, we conclude that
		\begin{multline*}
			g(\alpha, \beta, \delta) \ge h(\alpha, \delta) = \dfrac{\alpha + \sqrt{10 - \alpha^2 - 2\delta^2}}{2} - \sqrt{\dfrac{(\alpha - \sqrt{10 - \alpha^2 - 2\delta^2})^2}{4} + \delta^2} =\\
			\dfrac{1}{2}\left(\alpha + \sqrt{10 - \alpha^2 - 2\delta^2} - \sqrt{10 + 2\delta^2 - 2\alpha \sqrt{10 - \alpha^2 - 2\delta^2}}\right).
		\end{multline*}
		Now we rewrite the value $h(\alpha, \delta)$ in terms of $r = |z_1| = \sqrt{\alpha^2 + \delta^2}$ and $\delta$. It is easy to obtain the formula $h(\alpha, \delta) = \tilde h(r, \delta)$, where
		\begin{multline*}
			\tilde h(r, \delta) = \\ \dfrac{1}{2}\left(\sqrt{r^2 - \delta^2} + \sqrt{10 -  r^2 - \delta^2} - \sqrt{10 + 2\delta^2 - 2\sqrt{(r^2 - \delta^2)(10 - r^2 - \delta^2)}}\right).
		\end{multline*}
		It is easy to see that with fixed $r$ the function $\delta \mapsto \tilde h(r, \delta)$ achieves its minimum on $[-1,1]$ when $|\delta| = 1$. Thus,
		\begin{multline*}
			g(\alpha, \beta, \delta) \ge h(\alpha, \delta) = \tilde h(r, \delta) \le \tilde h(r, 1) = \\ \dfrac{1}{2}\left(\sqrt{r^2 - 1} + \sqrt{9 -  r^2} - \sqrt{12 - 2\sqrt{(r^2 - 1)(9 - r^2)}}\right).
		\end{multline*}
		Since $r = |z_1|$, we have proved the required estimate.
	\end{proof}

	Using the estimate from Lemma~\ref{lemPropertiesOfIntersectingCirclesInvolved}~\ref{CirclePropertyIii} we can now prove that we can slightly move the unit disc $\overline{\mathbb D}$ in order to avoid intersection with $F(z_1, z_2)$. In fact, $F(z_1, z_2)$ does not intersect one of discs $F_{\pm}$, where $F_{\pm}$ is the closed disc, such that the points $-1, 1, \pm \sqrt{3}i$ belong to the boundary of $F_{\pm}$ (see Fig.~\ref{figFMinus}).
	\FloatBarrier
	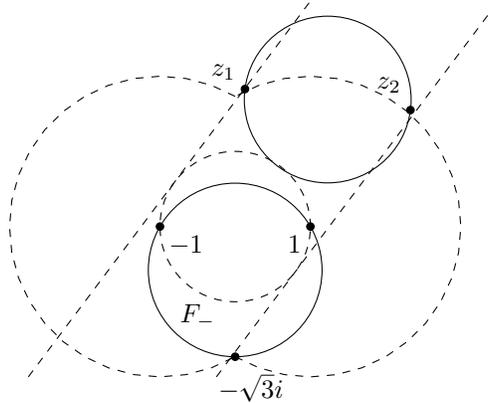
\begin{figure}
		\begin{tikzpicture}
				\draw[dashed] (-0.25, -2) -- (3.5, 3);
				\draw[dashed] (-2.75, -2) -- (1, 3);
				\draw[dashed] (0,0) circle[radius = 1];
				\draw[black, dashed, domain = -120:120] plot({1 + 2*cos(\x)}, {2*sin(\x)});
				\draw[black, dashed, domain = 60:300] plot({-1 + 2*cos(\x)}, {2*sin(\x)});
				\filldraw (0.13, 1.83) circle[radius = 0.05] node[above left] {$z_1$};
				\filldraw (2.33, 1.55) circle[radius = 0.05] node[above left, yshift=3] {$z_2$};
				\draw (1.23, 1.69) circle[radius = {sqrt(2.2*2.2 + 0.28*0.28)/2}];
				\filldraw (1, 0) circle[radius = 0.05] node[below left] {$1$};
				\filldraw (-1, 0) circle[radius = 0.05] node[below right] {$-1$};
				\filldraw (0, {-sqrt(3)}) circle[radius = 0.05] node[below, yshift=-3, xshift = 6] {$-\sqrt{3}i$};
				\draw (0,{-1/sqrt(3)}) circle[radius = {2/sqrt(3)}];
				\node at (-0.5, -1.2) {$F_-$};
			\end{tikzpicture}
			\caption{The disc $F_-$ drawn alongside the disc from Fig.~\ref{figExampleOfIntersectingDiscs}.}
			\label{figFMinus}
	\end{figure}
	
	\begin{lemma}\label{lemDeformedUnitCircle}
		Let $z_1,z_2 \in \CC \setminus \Omega$ be a pair of approximately collinear points such that $F(z_1,z_2) \cap \overline{\mathbb D} \ne \emptyset$. Assume that $|z_1 - z_2| \ge 2$ and let
		\[
		F_+ = \dfrac{1}{\sqrt{3}}i + \dfrac{2}{\sqrt{3}} \overline{\mathbb D},\;\;F_- = -\dfrac{1}{\sqrt{3}}i + \dfrac{2}{\sqrt{3}} \overline{\mathbb D}.
		\]
		Then, either $F(z_1,z_2) \cap F_+ = \emptyset$, or $F(z_1,z_2) \cap F_- = \emptyset$.
	\end{lemma}
	\begin{proof}
		Without loss of generality we may assume that $\mathrm{Im}(z_1), \mathrm{Im}(z_2) > 0$ (by Lemma~\ref{lemPropertiesOfIntersectingCirclesInvolved}~\ref{CirclePropertyIi}) and that $|z_1| \le |z_2|$. We shall prove that $F(z_1,z_2)$ does not intersect $F_-$.
		
		Consider $t \in [-\pi, \pi]$ such that $a = e^{it}$ satisfies $z_1/a,z_2/a \in \Sigma$ and $\mathrm{Im}(z_1/a)  = -\mathrm{Im}(z_2/a)$. Now it suffices to prove that the discs $a^{-1}F_-$ and $a^{-1}F(z_1, z_2) = F(z_1/a, z_2/a)$ do not intersect. Since the center of the disc $F(z_1/a, z_2/a)$ is a positive real number, from Lemma~\ref{lemPropertiesOfIntersectingCirclesInvolved}~\ref{CirclePropertyIii} for all $z \in F(z_1/a, z_2/a)$ we have the estimate
		\[
		\mathrm{Re}(z) \ge \phi(r) = \dfrac{1}{2}\left(\sqrt{r^2 - 1} + \sqrt{9 - r^2} - \sqrt{12 -2\sqrt{(r^2 - 1)(9 - r^2)}}\right),
		\]
		where $r = |z_1|$. On the other hand we can calculate the upper bound for $\mathrm{Re}(z)$, if $z \in a^{-1}F_-$. Indeed, 
		\[
		a^{-1}F_- = -\dfrac{a^{-1}}{\sqrt{3}}i + \dfrac{2}{\sqrt{3}}\overline{\mathbb D} = -\dfrac{\overline{a}}{\sqrt{3}}i + \dfrac{2}{\sqrt{3}}\overline{\mathbb D}.\]
		Therefore, if $z \in a^{-1}F_-$, then
		\[
		\mathrm{Re}(z) \le \mathrm{Re}\left(-\dfrac{\overline{a}}{\sqrt{3}}i\right) + \dfrac{2}{\sqrt{3}} = \dfrac{-\mathrm{Im}(a) + 2}{\sqrt{3}}.
		\]
		It remains to estimate $\mathrm{Im}(a)$. We can use the fact that $z_1/a \in \Sigma$. If $t_0 \in [-\pi, \pi]$ is chosen such that $e^{-it_0}z_1$ is a positive real number, then by Lemma~\ref{lemAnglesOfSmallAbsoluteValuesOutsideOmega} $t_0 \in [q(r), \pi - q(r)]$. Then by Lemma~\ref{lemArcsin}~\ref{arcsinii} we find that $a = e^{it_0 + is}$, where $|s| \le \arcsin{1/r}$. From this we infer that $a = e^{it}$, where $t \in [q(r) - \arcsin(1/r), \pi - q(r) + \arcsin(1/r)]$. It is straightforward to verify that for $r \in [\sqrt{3}, \sqrt{5}]$ we have $\arcsin(1/r) < q(r)$. Therefore, we can estimate
		\begin{multline*}
			\mathrm{Im}(a) \ge \sin(q(r) - \arcsin(1/r)) = \sin(q(r))\cos(\arcsin(1/r)) - \cos(q(r))\dfrac{1}{r} = \\ \dfrac{(r^2 - 1)\sqrt{9 - r^2} - r^2 + 3}{2r^2}.
		\end{multline*}
		Finally, it follows that
		\[
		\mathrm{Re}(z) \le \psi(r) = \dfrac{5r^2 - 3 - (r^2 - 1)\sqrt{9 - r^2}}{2r^2\sqrt{3}}
		\]
		for all $z \in a^{-1}F_-$. Thus, to finish the proof it suffices to prove that $\psi(r) < \phi(r)$ for all $r \in [\sqrt{3}, \sqrt{5}]$.
		
		The inequality $\psi(r) < \phi(r)$ for all $r \in [\sqrt{3}, \sqrt{5}]$ can be proved by verifying the following statements.
		\begin{enumerate}
			\item $\phi$ and $\psi$ are nondecreasing on $[\sqrt{3}, \sqrt{5}]$.
			\item $\psi(\sqrt{5}) < \phi(2)$ and $\psi(2) < \phi(\sqrt{3})$.
		\end{enumerate}
		Indeed, if these statements hold, then for $r \in [\sqrt{3}, 2]$ we get that $\psi(r) \le \psi(2) < \phi(\sqrt{3}) \le \phi(r)$. Similarly, for $r \in [2, \sqrt{5}]$ we have $\psi(r) \le \psi(\sqrt{5}) < \phi(2) \le \phi(r)$. The second statement is proved by a trivial calculation, while the monotonicity can be proved by differentiating. We omit these verifications.
	\end{proof}
	
	Due to Lemma~\ref{lemDeformedUnitCircle}, given a distinguished set $E = \{z_1, z_2, z_3\}$ such that $z_1$ and $z_2$ are approximately collinear and $F(z_1, z_2) \cap \overline{\mathbb D} \ne \emptyset$, we have two disjoint closed disks $D_1 = F_{\pm}$ and $D_2 = F(z_1, z_2)$ such that $-1, 1 \in D_1$ and $z_1, z_2 \in D_2$. Thus, it remains to prove that $z_3 \notin \conv(D_1 \cup D_2)$. In fact, we shall prove a stronger statement that $z_3 \notin \conv(F(z_1, z_2) \cup \sqrt{3}\cdot\overline{\mathbb D})$ provided that $|z_1 - z_2| < 2\sqrt{2}$ (note that $F_{\pm} \subset \sqrt{3}\cdot\overline{\mathbb D}$). In view of Lemma~\ref{lemPropertiesOfIntersectingCirclesBasic}~\ref{CirclePropertyiii} the inequality $|z_1 - z_2| < 2\sqrt{2}$ always holds, if $z_1, z_2 \ne \pm \sqrt{3}i$.
	
	At first we need the following general statement about coverings by convex sets.
	
	\begin{lemma}\label{lemConvexCovering}
		Assume that $K \subset \RR^n$ is a compact set and that $\{C_\alpha\}_{\alpha \in I}$ is an arbitrary family of convex subsets of $\RR^n$. Assume that $\partial K \subset \bigcup_{\alpha \in I} C_\alpha$ and that $\bigcap_{\alpha \in I} C_\alpha \ne \emptyset$. Then $K \subset \bigcup_{\alpha \in I} C_\alpha$.		
	\end{lemma}
	\begin{proof}
		We need to prove that for all $z \in \mathrm{Int}(K)$ we have $z \in \bigcup_{\alpha \in I} C_\alpha$. Let $a \in \bigcap_{\alpha \in I} C_\alpha$ and consider $z(t) = (1-t)a + tz$ for $t \in \RR$. Clearly, $z(1) = z \in \mathrm{Int}(K)$ and $z(t) \notin K$ for large $t$, so there exists $t_0 > 1$ such that $z(t_0) \in \partial K$. Therefore, there is $\alpha_0 \in I$ such that $z(t_0) \in C_{\alpha_0}$. Since $a = z(0) \in C_{\alpha_0}$ and $t_0 > 1$, by convexity we conclude that $z = z(1) \in C_{\alpha_0}$.
	\end{proof}
	
	\begin{lemma}\label{lemNotInConvIfIntersects}
		Let $z_1, z_2 \in \CC$ satisfy the inequalities $|z_1|, |z_2| \ge \sqrt{3}$ and $2 \le |z_1 - z_2| < 2\sqrt{2}$. Moreover, assume that $F(z_1,z_2) \cap \overline{\mathbb D} \ne \emptyset$. Then \[\conv(F(z_1,z_2) \cup \sqrt{3}\cdot\overline{\mathbb D}) \subset \sqrt{5}\mathbb D \cup (z_1 + 2\mathbb D) \cup (z_2 + 2\mathbb D).\]
	\end{lemma}
	The number $\sqrt{5}$ does not play any important role in the proof and, definitely, can be slightly decreased, if one uses optimal estimates. Nevertheless, the radius $\sqrt{5}$ is sufficient for our purposes. Also, the statement of this lemma is illustrated in Fig.~\ref{figConvCovering}.
	\FloatBarrier
	\begin{figure}
		\begin{tikzpicture}
			\draw (0,0) circle[radius = {sqrt(3)}];
			\draw[dashed] (0,0) circle[radius = {sqrt(5)}];
			\draw (2, 0) circle[radius = 1.4];
			\filldraw ({2 + sqrt(1 - 0.81)*1.4}, {0.9*1.4}) circle[radius = 0.05] node[above right] {$z_1$};
			\filldraw ({2 - sqrt(1 - 0.81)*1.4}, {-0.9*1.4}) circle[radius = 0.05] node[above right] {$z_2$};
			\draw[dashed] ({2 + sqrt(1 - 0.81)*1.4}, {0.9*1.4}) circle[radius = 2];
			\draw[dashed] ({2 - sqrt(1 - 0.81)*1.4}, {-0.9*1.4}) circle[radius = 2];
			\draw ({sqrt(3)*(sqrt(3) - 1.4)/2}, {sqrt(3)*sqrt(1 - (sqrt(3) - 1.4)*(sqrt(3) - 1.4)/4)}) -- ({2 + 1.4*(sqrt(3) - 1.4)/2}, {1.4*sqrt(1 - (sqrt(3) - 1.4)*(sqrt(3) - 1.4)/4)});
			\draw ({sqrt(3)*(sqrt(3) - 1.4)/2}, {-sqrt(3)*sqrt(1 - (sqrt(3) - 1.4)*(sqrt(3) - 1.4)/4)}) -- ({2 + 1.4*(sqrt(3) - 1.4)/2}, {-1.4*sqrt(1 - (sqrt(3) - 1.4)*(sqrt(3) - 1.4)/4)});
			\node at (-0.5, 1.95) {$\sqrt{5}\DD$};
			\node at (-0.9, 0.7) {$\sqrt{3}\cdot\overline{\DD}$};
			\node at (2.9, 0.6) {$F$};
			\node at (3, 2.7) {$z_1 + 2\DD$};
			\node at (2, -2.7) {$z_2 + 2\DD$};
		\end{tikzpicture}
		\caption{The illustration of the statement of Lemma~\ref{lemNotInConvIfIntersects} (where we denoted $F(z_1,z_2)$ as $F$ for convenience).}
		\label{figConvCovering}
	\end{figure}
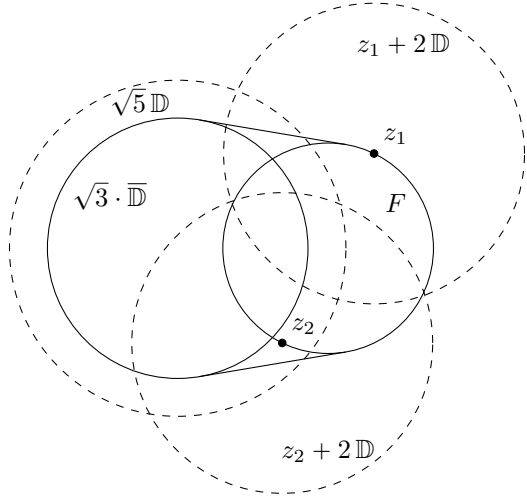
	\begin{proof}
		If $z_1 + z_2 = 0$, then the statement is trivial, since in this case $F(z_1,z_2) \subset \sqrt{2}\cdot \overline{\mathbb D}$. Thus, we may assume that $c = (z_1 + z_2/2$ is a positive real number. If $r = |z_1 - z_2|/2$, then $r < \sqrt{2}$, $F(z_1,z_2) = c + r\overline{\mathbb D}$, and $c \le r + 1$ (since $F(z_1, z_2) \cap \overline{\mathbb D} \ne \emptyset$). It follows that $c \le \sqrt{2} + 1$. We claim that $s = \min\{c, 2\}$ belongs to $\sqrt{5}\mathbb D \cap (z_1 + 2\mathbb {D}) \cap (z_2 + 2\mathbb D)$. Indeed, $0 < s \le 2 < \sqrt{5}$, so $s \in \sqrt{5}\mathbb D$. Moreover, $|s - z_1| \le |s - c| + |c - z_1| < (\sqrt{2} + 1 - 2) + \sqrt{2} = 2\sqrt{2} - 1 < 2$. Similarly $|s - z_2| < 2$, so the statement is proved. Thus, $\sqrt{5}\mathbb D \cap (z_1 + 2\mathbb {D}) \cap (z_2 + 2\mathbb D) \ne \emptyset$, so by Lemma~\ref{lemConvexCovering} it suffices to prove that
		\[
		\partial \conv(F(z_1,z_2) \cup \sqrt{3}\cdot\overline{\mathbb D}) \subset \sqrt{5}\mathbb D \cup (z_1 + 2\mathbb D) \cup (z_2 + 2\mathbb D).
		\]
		
		Before proceeding further we find parametric representations for $z_1$ and $z_2$, and in the meantime obtain some restrictions on $r$ and $c$. Without loss of generality we may assume that $\mathrm{Im}(z_1) \ge 0$, so $\mathrm{Im}(z_2) \le 0$. Thus, there is $\gamma \in [0,\pi]$ such that $z_1 = c + re^{i\gamma}$. From the inequality $|z_1| \ge \sqrt{3}$ we infer that
		\[
		\cos(\gamma) \ge \dfrac{3 - r^2 - c^2}{2cr}.
		\]
		Moreover, since $z_2 = c - re^{i\alpha}$ and $|z_2| \ge \sqrt{3}$, we have the inequality
		\[
		\cos(\gamma) \le \dfrac{r^2 + c^2 - 3}{2cr}.
		\]
		Thus, 
		\[
		|\cos(\gamma)| \le \dfrac{r^2 + c^2 - 3}{2cr},
		\]
		and, in particular, we conclude that $r^2 + c^2 \ge 3$. 
		
		Now we introduce $a(c,r) \in \CC$ such that $|a(c,r)| = 1$, $\mathrm{Im}(a(c,r)) > 0$, and $\mathrm{Re}(a(c,r)) = (\sqrt{3} - r)/c$ (such $a$ exists because $c^2 \ge 3 - r^2 > \sqrt{3} - r$). Moreover, if $|b| = 1$ and $\mathrm{Re}(b) = (\sqrt{3} - r)/c$, then either $b = a(c,r)$, or $b = \overline{a(c,r)}$. Hence, by Lemma~\ref{lemConvexHullDescription} we have the inclusion
		\[
		\partial \conv(F(z_1,z_2) \subset \sqrt{3}\cdot \overline{D} \cup F(z_1, z_2) \cup \{p_{c,r}(t):t \in (0,1)]\} \cup \{\overline{p(t)}: t \in (0,1)]\},
		\]
		where $p_{c,r}(t) = ct + (rt + \sqrt{3}(1 - t))a(c,r)$. Clearly, $\sqrt{3}\cdot \overline{\mathbb D} \subset \sqrt{5}\mathbb D$ and $F(z_1, z_2) \subset (z_1 + 2\mathbb D) \cup (z_2 + 2\mathbb D)$ (the latter inclusion holds, because the radius $r$ of $F(z_1,z_2)$ is strictly less that $\sqrt{2}$). Thus, the problem reduces to the inclusion
		\begin{equation}\label{eqBoundaryReduction}
			\left\{p_{c,r}(t):t \in (0,1)]\right\} \cup \left\{\overline{p(t)}: t \in [0,1]\right\} \subset \sqrt{5}\mathbb D \cup (z_1 + 2\mathbb D) \cup (z_2 + 2\mathbb D).
		\end{equation}
		
		Now it suffices to prove the following statement. Given any $\alpha \in [0, \pi]$ such that $|\cos(\alpha)| \le (c^2 + r^3 - 3)/(2cr)$ we have the inclusion
		\begin{equation}\label{eqBoundaryReductionFurther}
			\left\{p_{c,r}(t):t \in (0,1)]\right\} \subset \sqrt{5}\mathbb D \cup (c + re^{i\alpha} + 2\mathbb D).
		\end{equation}
		Indeed, if this inclusion holds, then by taking conjugation we conclude that
		\[
		\left\{\overline{p_{c,r}(t)}:t \in (0,1)]\right\} \subset \sqrt{5}\mathbb D \cup (c - re^{i\alpha} + 2\mathbb D)
		\]
		for any $\alpha \in [0, \pi]$ such that $|\cos(\alpha)| \le (c^2 + r^3 - 3)/(2cr)$. Thus,~\eqref{eqBoundaryReduction} would follow from these inclusion with $\alpha = \gamma$.
		
		By a straightforward calculation it is easy to obtain the formula
		\[
		|p_{c,r}(t)| = \sqrt{3 + t^2\left(c^2 - \left(\sqrt{3} - r\right)^2\right)}.
		\]
		Thus, $t \mapsto |p_{c,r}(t)|$ is strictly increasing on the ray $[0, +\infty)$ and, in particular, for 
		\[
		t_0(c,r) = \dfrac{\sqrt{2}}{\sqrt{c^2 - \left(\sqrt{3} - r\right)^2}}
		\]
		we have the inequalities $|p_{c,r}(t)| < \sqrt{5}$ if $t \in [0, t_0(c,r))$ and $|p_{c,r}(t)| > \sqrt{5}$ if $t > t_0(c, r)$. Thus, we have the inclusion
		\[
		\left\{\overline{p_{c,r}(t)}:t \in (0,t_0(c,r))]\right\} \subset  \sqrt{5}\mathbb D.
		\]
		In particular, if $t_0(c,r) \ge 1$, then~\eqref{eqBoundaryReductionFurther} holds. Thus, we may further assume that $t_0(c,r) < 1$, i.e. $c^2 - (\sqrt{3} - r)^2 > 2$, and the problem reduces to proving the inclusion
		\begin{equation}\label{eqBoundaryReductionEvenFurther}
			\left\{\overline{p_{c,r}(t)}:t \in (t_0(c, r),1)]\right\} \subset (c - re^{i\alpha} + 2\mathbb D)
		\end{equation}
		for all $\alpha \in [0, \pi]$ such that $|\cos(\alpha)| \le (c^2 + r^3 - 3)/(2cr)$. Moreover, it suffices to consider the parameters $c$ and $r$ that satisfy the inequalities
		\begin{equation}\label{eqRestrictionsForRC}
			1 \le r < \sqrt{2},\;\;0 < c \le r + 1,\;\;c^2 - \left(\sqrt{3} - r\right)^2 > 2.
		\end{equation}
		It should be noted that the inequality $c^2 + r^2 \ge 3$, that we proved earlier, easily follows from~\eqref{eqRestrictionsForRC}, as $r \ge 1$ and $c^2 > 2$ (from the third inequality).
		
		Now we need the following {\it{auxiliary statement}}. Assume that $z \in \CC$ and $\mathrm{Re}(z) \ge 0$. Also let $0 \le \phi \le \psi \le \pi$, $x \in \RR$, and $\rho > 0$. Then
		\[
		\max\limits_{\alpha \in [\phi, \psi]} |z - x - re^{it}| = \max\left\{|z - x - re^{i\phi}|, |z - x - re^{i\psi}|\right\}.
		\]
		That, is we are given an arc $\Gamma = \{x + re^{it}: t \in [\phi, \psi]\}$ of the circle $x + r\mathbb T$ such that $\Gamma$ is contained in the upper-half plane, and a point $z$ that also lies in the upper half-plane. Then the maximal distance from $z$ to a point $y \in \Gamma$ is attained at an endpoint of $\Gamma$. This statement is trivial if $z = x$, and otherwise can be proved by an elementary calculation that shows that there is no local maxima of the function $t \mapsto |z - x - re^{it}|$ on $(\phi, \psi)$, if $(\phi, \psi) \subset [0, \pi]$.
		
		From the previous paragraph we can further reduce our problem. That is, let $w(c,r) = \arccos((r^2 + c^2 - 3)/(2cr))$ and consider $y_+(c,r) = c + re^{iw(c,r)}$ and $y_-(c,r) = c + re^{i(\pi -w(c,r))}$. The complex numbers $y_+$ and $y_-$ are the endpoints of the arc
		\[
		\left\{z = c + re^{i\alpha}: \alpha \in [0, \pi], |\cos(\alpha)| \le \dfrac{r^2 + c^2 - 3}{2cr} \right\}.
		\]
		With these definitions it is easy to see that~\eqref{eqBoundaryReductionEvenFurther} holds (for all $\alpha$ in the indicated interval) if the inequalities
		\begin{equation}\label{eqReductionToFourDistances}
			|p_{c,r}(t_0(c,r)) - y_{\pm}(c,r)| < 2, \;\;|p_{c,r}(1) - y_{\pm}(c,r)| < 2
		\end{equation}
		hold for all $c$ and $r$ that satisfy~\eqref{eqRestrictionsForRC}.
		
		The verification of~\eqref{eqReductionToFourDistances} is elementary and is based on several estimations that can be obtained by finding extrema of functions of variables $(c,r)$ in the region defined in~\eqref{eqRestrictionsForRC}. However, the proof is rather long, so we shall omit the calculations of derivatives and determination of their signs; in what follows we shall only indicate the type of monotonicity of the corresponding functions.
		
		Let us denote by $P$ the set of all $(c,r)$ such that~\eqref{eqRestrictionsForRC} holds. It is easy to see that $\mathrm{Re}(a(c,r)) = (\sqrt{3} - r)/c$ decreases with respect to $c$. Thus, $\mathrm{Re}(a(c,r)) \ge \mathrm{Re}(a(r+1,r)) = (\sqrt{3} - r)/(r + 1)$. This function also decreases with respect to $r$, so $\mathrm{Re}(a(c,r)) \ge \mathrm{Re}(a(\sqrt{2} + 1, \sqrt{2})) = (\sqrt{3} - \sqrt{2})/(\sqrt{2} + 1)$. Now we can similarly estimate that, since $c \ge 3 - r^2$ and $r \ge 1$ for $(c,r) \in P$ we have
		\[
		\mathrm{Re}(a(c,r)) \le \dfrac{\sqrt{3} - r}{3 - r^2} = \dfrac{1}{\sqrt{3} + r} \le \dfrac{1}{1 + \sqrt{3}}.
		\]
		Thus, we conclude that 
		\[
		\dfrac{\sqrt{3} - \sqrt{2}}{\sqrt{2} + 1} \le \mathrm{Re}(a(c,r)) \le \dfrac{1}{1 + \sqrt{3}},\text{ for all }(c,r) \in P.
		\]
		Now let $h(c,r) = (r^2 + c^2 - 3)/(2cr)$ (so $w(c,r) = \arccos(h(c,r))$). We can similarly find the bounds on $h$ in the region $P$. It is straightforward to verify that $h$ increases with respect to $c$ in $P$, so $h(c,r) \le h(r+1,r) = 1 - 1/(r(r+1))$. Thus function, clearly, increases with $r$, so $h(c,r) \le h(\sqrt{2} + 1, \sqrt{2}) = 1/\sqrt{2}$. To find the lower bound, at first note that $h(c,r) \ge h\left(\sqrt{2 + (\sqrt{3} - r)^2}, r\right)$, since $h$ is increasing with respect to $c$. The resulting function can be proved to be increasing with respect to $r$, so
		\[
		h(c,r) \ge h\left(\sqrt{2 + (\sqrt{3} - 1)^2}, 1\right) = \dfrac{1 - 3 + 2 + (\sqrt{3} - 1)^2}{2\sqrt{2 + (\sqrt{3} - 1)^2}} = \dfrac{2 - \sqrt{3}}{\sqrt{6 - 2\sqrt{3}}}.
		\]
		In fact, we shall settle for a less sharp bounds for $h(c,r)$ and $\mathrm{Re}(a(c,r))$. That is, the inequalities
		\[
		\dfrac{\sqrt{3} - \sqrt{2}}{\sqrt{2} + 1} \le \dfrac{2 - \sqrt{3}}{\sqrt{6 - 2\sqrt{3}}} \le \dfrac{1}{1 + \sqrt{3}} \le \dfrac{1}{\sqrt{2}}
		\]
		and the foregoing bounds imply that
		\begin{equation}\label{eqBoundsForAandH}
			\mathrm{Re}(a(c,r)), h(c,r) \in \left[\dfrac{\sqrt{3} - \sqrt{2}}{\sqrt{2} + 1}, \dfrac{1}{\sqrt{2}}\right]
		\end{equation}
		
		Now we can easily prove the second pair of inequalities in~\eqref{eqReductionToFourDistances}. Indeed, if $\alpha, \beta \in [\pi/4, 3\pi/4]$, then 
		\begin{equation}\label{eqDistanceOnCircleRightAngle}
			|re^{i\alpha} - re^{i\beta}| = r|e^{i(\alpha - \beta)} - 1| = r\sqrt{(1 - \cos(\alpha - \beta)^2 + \sin(\alpha - \beta)^2)} \le r\sqrt{2} < 2,
		\end{equation}
		since $|\alpha - \beta| \le \pi/2$ (and, therefore, $\cos(\alpha - \beta) \ge 0$). Now from~\eqref{eqBoundsForAandH} it follows that $w(c,r) \in [\pi/4, \pi/2]$. Similarly, by~\eqref{eqBoundsForAandH}, $a(c,r)$ can be written as $a(c,r) = e^{i\alpha(c,r)}$, where $\alpha(c,r) \in [\pi/4, \pi/2]$. Thus, from~\eqref{eqDistanceOnCircleRightAngle} it follows that 
		\[
		|y_+(c,r) - p_{c,r}(1)| = |c + re^{iw(c,r)} - c - ra| = |re^{iw(c,r)} - re^{i\alpha(r,c)}| < 2.
		\]
		Similar inequality holds for $y_-(c,r)$, since $y_-(c,r) = c + re^{i(\pi - w(c,r))}$ and $\pi - w(c,r) \in [\pi/2, 3\pi/4]$.
		
		Now we prove that $|y_-(c,r) - p_{c,r}(t_0(c,r))| < 2$ for all $(c,r) \in P$. It is clear that $|y_-(c,r)| = \sqrt{3}$, so we can write $y_-(c,r) = \sqrt{3}b(c,r)$, where $|b| = 1$. By a straightforward calculation we have
		\[
		\mathrm{Re}(b(c,r)) = \dfrac{1}{\sqrt{3}} \mathrm{Re}(y_-) = \dfrac{3 + c^2 - r^2}{2\sqrt{3}c}.
		\]
		From this formula it is easy to verify that $\mathrm{Re}(a(c,r)) \ge \mathrm{Re}(b(c,r))$ for all $c,r \in P$. Thus, if $b = e^{i\beta(c,r)}$, where $\beta(c,r) \in [0, \pi]$, then we have the inequalities $0 \le \beta(c,r) \le \alpha(c,r) \le \pi/2$. Therefore, $h_{c,r}(0) = \sqrt{3}a(c,r)$ lies on the arc $\{z = \sqrt{3}e^{it}: t \in [\beta(c,r), \pi/2]\}$, so by auxiliary statement we can prove that $|y_-(c,r) - h_{c,r}(0)| < 2$, if we prove that $|y_-(c,r) - \sqrt{3}i| < 2$ (note that the point $y_-(c,r)$ coincides with the other endpoint of the arc). Also, it is clear that $|y_-(c,r) - \sqrt{3}i|$ is maximal when $\mathrm{Re}(b(c,r))$ is maximal. It is easy to verify that $\mathrm{Re}(b(c,r))$ increases with respect to $c$, so $\mathrm{Re}(b(c,r)) \le \mathrm{Re}(b(r+1,r))$. The resulting function decreases with $r$, so 
		\[
		\mathrm{Re}(b(c,r)) \le \mathrm{Re}(b(2,1)) = \dfrac{3 + 4 - 1}{2\sqrt{3}2} = \dfrac{\sqrt{3}}{2}.
		\]
		Therefore, 
		\[
		|\sqrt{3}i - y_{-}(c,r)| \le \left|\sqrt{3}i - \sqrt{3}\left(\dfrac{\sqrt{3}}{2} + \dfrac{i}{2}\right)\right| = \left|\dfrac{3}{2} + \dfrac{\sqrt{3}}{2}i\right| = \sqrt{\dfrac{9}{4} + \dfrac{3}{4}} = \sqrt{3} < 2.
		\]
		Thus, we have verified that $|p_{c,r}(0) - y_-(c,r)| < 2$ for all $(c,r) \in P$. Since we have already established that $|p_{c,r}(1) - y_-(c,r)| < 2$, we have the same inequality for all intermediate points, in particular, $|p_{c,r}(t_0(c,r)) - y_-(c,r)| < 2$.
		
		It remains to prove the inequality $|p_{c,r}(t_0(c,r)) - y_+(c,r)| < 2$ for all $(c,r) \in P$. Fortunately, it suffices to find sufficiently good bounds for $|p_{c,r}(t_0(c,r)) - p_{c,r}(1)|$ and $|p_{c,r}(1) - y_+(c,r)|$ and combine them with the triangle inequality. To calculate the first bound note that
		\[
		|p_{c,r}(t) - p_{c,r}(s)| = |t - s| \sqrt{c^2 - (\sqrt{3} - r)^2}
		\]
		for all $t, r \in \RR$. In particular,
		\[
		|p_{c,r}(t_0(c,r)) - p_{c,r}(1)| = |t_0(c,r) - 1|\sqrt{c^2 - (\sqrt{3} - r)^2} = \sqrt{c^2 - (\sqrt{3} - r)^2} - \sqrt{2}.
		\]
		The last expression, clearly increases with respect to both $c$ and $r$, so 
		\begin{multline}\label{eqBoundForPartOfTangentLine}
			|p_{c,r}(t_0(c,r)) - p_{c,r}(1)| \le |p_{\sqrt{2} + 1,\sqrt{2}}(t_0(\sqrt{2} + 1,\sqrt{2})) - p_{\sqrt{2} + 1,\sqrt{2}}(1)|  = \\ \sqrt{2}\left(\sqrt{\sqrt{2} + \sqrt{6} - 1} - 1\right).
		\end{multline}
		To find a good bound for $|p_{c,r}(1) - y_+(c,r)|$ we need to only slightly modify the previous argument (that gave $2$ as a bound). From~\eqref{eqBoundsForAandH}, it follows that 
		\begin{multline*}
			|\alpha(c,r) - w(c,r)| \le \arccos\left(\dfrac{\sqrt{3} - \sqrt{2}}{\sqrt{2} +1}\right) - \dfrac{\pi}{4} = \\ \arccos\left(\dfrac{\sqrt{3} - \sqrt{2} + \sqrt{2}\sqrt{\sqrt{6} + \sqrt{2} - 1}}{\sqrt{2}(\sqrt{2} + 1)}\right)
		\end{multline*}
		From this we can conclude that
		\begin{multline*}
			|p_{c,r}(1) - y_+(c,r)| = r|e^{i\alpha(c,r)} - e^{iw(c,r)}| = r |e^{i(\alpha(c,r) - w(c,r))} - 1| = \\ r\sqrt{2 - 2\cos(\alpha(c,r) - w(c,r))} \le
			2\sqrt{1 - \dfrac{\sqrt{3} - \sqrt{2} + \sqrt{2}\sqrt{\sqrt{6} + \sqrt{2} - 1}}{\sqrt{2}(\sqrt{2} + 1)}}
		\end{multline*}
		Now it remains to note that
		\[
		\sqrt{2}\left(\sqrt{\sqrt{2} + \sqrt{6} - 1} - 1\right) < 1,\;\;2\sqrt{1 - \dfrac{\sqrt{3} - \sqrt{2} + \sqrt{2}\sqrt{\sqrt{6} + \sqrt{2} - 1}}{\sqrt{2}(\sqrt{2} + 1)}} < 1,
		\]
		so
		\[
		|y_+(c,r) - p_{c,r}(t_0(c,r))| \le |p_{c,r}(t_0(c,r)) - p_{c,r}(1)| + |p_{c,r}(1) - y_+(c,r)|< 1 + 1 = 2.
		\]
	\end{proof}
	
	\begin{proposition}\label{propIfNotSqrt3ThenSplittable}
		Assume that $E$ is a distinguished set, that does not contain $\pm \sqrt{3}i$. Then $\{-1, 1, \infty\} \cup E$ is splittable.
	\end{proposition}
	\begin{proof}
		At first we note, that due to Propositions~\ref{propSplittingIfAsterisk} and~\ref{propIfWellSeparatedThenSplittable} we may assume that $E$ is not splittable by a strip and is not well-separated from zero. Choose the enumeration $E = \{z_1,z_2,z_3\}$ such that $z_1$ and $z_2$ are approximately collinear, $F(z_1, z_2) \cap \overline{\mathbb D} \ne \emptyset$, and $|z_1| \le |z_2|$. By Lemma~\ref{lemDeformedUnitCircle} either $F(z_1,z_2) \cap F_+ = \emptyset$, or $F(z_1,z_2) \cap F_- = \emptyset$. Let $F$ be one of the discs $F_{\pm}$ that does not intersect $F(z_1,z_2)$. Then, $-1,1 \in F$, so by Lemma~\ref{lemFinishingSplitting} it remains to show that $z_3 \notin \conv(F \cup F(z_1,z_2))$. Since $z_1 \ne \pm \sqrt{3}i$ by Lemma~\ref{lemPropertiesOfIntersectingCirclesBasic}~\ref{CirclePropertyiii} we have the inequality $|z_1 - z_2| < 2\sqrt{2}$. Thus, by Lemma~\ref{lemNotInConvIfIntersects} $ \conv(F \cup F(z_1,z_2)) \subset \sqrt{5}\mathbb D \cup (z_1 + 2\mathbb D) \cup (z_2 + 2\mathbb D)$. Therefore, the proof reduces to checking that $z_3 \notin \sqrt{5}\mathbb D \cup (z_1 + 2\mathbb D) \cup (z_2 + 2\mathbb D)$. Clearly, $z_3 \notin (z_1 + 2\mathbb D) \cup (z_2 + 2\mathbb D)$, since $E$ is distinguished. Thus, it remains to prove that $|z_3| \ge \sqrt{5}$. Assume the contrary, i.e. $|z_3| < \sqrt{5}$. From Lemma~\ref{lemPropertiesOfIntersectingCirclesBasic}~\ref{CirclePropertyi} it follows that $|z_1| \le \sqrt{5}$. Moreover, if $z \in \CC \setminus \Omega$ and $|z| \le \sqrt{5}$, then $|\mathrm{Re}(z)| \le 1$. Thus, $|\mathrm{Re}(z_1)|, |\mathrm{Re}(z_3)| \le 1$. If the numbers $\mathrm{Im}(z_1)$ and $\mathrm{Im}(z_3)$ have the opposite signs, then $E$ is splittable by a strip in view of Lemma~\ref{lemIfNotAsteriskThenNoOppositeStrips}. Thus, we may also assume that $\mathrm{Im}(z_1), \mathrm{Im}(z_2) > 0$ (the other case is handled similarly).
		
		Now to finish the proof it remains to verify the following {\it{auxiliary statement}}. If $w = \alpha + i\beta \in \CC \setminus \Omega$, $\mathrm{Im}(w) > 0$, and $|w| \le \sqrt{5}$ (resp. $|w| < \sqrt{5}$), then $|w - 2i| \le 1$ (resp. $|w - 2i| < 1|$). Applying this fact to $z_1$ and $z_3$ (with assumptions of the previous paragraph) we immediately get the contradiction with the fact that $E$ is distinguished, since $|z_1 - z_3| \le |z_1 - 2i| + |z_3 - 2i| < 1 + 1 = 2$. To prove the auxiliary statement at first assume that $|w| \le \sqrt{5}$ and $\beta = \mathrm{Im}(w) \ge 2$. Then $|w - 2i|^2 = (\beta - 2)^2 + \alpha^2 = \beta^2 - 4\beta + 4 + \alpha^2 \le 5 - 8 + 4 = 1$. Moreover, this inequality is strict, if $\alpha^2 + \beta^2 < 5$. Now assume that $|\alpha| \le 1$ and $0 < \beta < 2$. We claim that if $|w - 2i| \ge 1$, then either $|w - 1| < 2$, or $|w + 1| < 2$. Clearly, we may assume $\alpha \ge 0$ without loss of generality. Then we have the inequality $(\beta - 2)^2 \ge 1 - \alpha^2$, so $2 - \beta \ge \sqrt{1 - \alpha^2} \ge 1 - \alpha$. Therefore, $\beta \le 1 + \alpha$. Now we have the inequality $|w - 1|^2 = \beta^2 + (\alpha - 1)^2 \le (\alpha + 1)^2 + (\alpha - 1)^2 = 2\alpha^2 + 2$. If $\alpha < 1$, then $2 + 2\alpha^2 < 4$, so $|w - 1| < 2$. If $\alpha = 1$, then $|w - 1| = \beta < 2$. Thus, anyway, $|w - 1| < 2$.
	\end{proof}

	In order to finish the proof we need the last auxiliary lemma.
	
	\begin{lemma}\label{lemIfSqrt3AndNotSplittableByAStripThenHalfPlane}
 		Assume that $E \subset \CC$ is a distinguished set such that $\sqrt{3}i \in E$ and $E$ is not splittable by a strip. Then $\mathrm{Im}(z) > -1$ for all $z \in E$. 
	\end{lemma}
	\begin{proof}
		This can be verified by a slight modification of the proof of Lemma~\ref{lemPropertiesOfIntersectingCirclesBasic}~\ref{CirclePropertyiv}. That is, let $z,w$ be the remaining points in $E$, i.e. $E = \{\sqrt{3}i, z, w\}$. In view of Proposition~\ref{propIfSharpNotAsteriskThenStrips} we may without loss of generality assume that $\sqrt{3}i$ and $z$ are approximately collinear. 
		
		We claim that $\mathrm{Im}(z) > 1$. If $\mathrm{Im}(z) \le -1$, then $\sqrt{3}i$ and $z$ are not approximately collinear by Lemma~\ref{lemIfLargeImaginaryOppositeSignsThenNotApprox}. Now assume that $|\mathrm{Im}(z)| \le 1$. Without loss of generality we may assume that $\mathrm{Re}(z) \ge 0$. Then, since $z \notin \Omega$, it follows that $\mathrm{Re}(z) \ge 1 + \sqrt{3}$. Now consider 
		\[
		a = \dfrac{3 + 2\sqrt{3}}{5 + 2\sqrt{3}} + i\dfrac{2 + 2\sqrt{3}}{5 + 2\sqrt{3}}.
		\]
		Then it is easily verified that $|a| = 1$, and that
		\[
		\mathrm{Im}\left(\dfrac{i\sqrt{3}}{a}\right) = \sqrt{3}\dfrac{3 + 2\sqrt{3}}{5 + 2\sqrt{3}} = \dfrac{3\sqrt{3} + 6}{2\sqrt{3} + 5} > 1.
		\]
		On the other hand
		\begin{multline*}
			\mathrm{Im}\left(\dfrac{z}{a}\right) = \dfrac{1}{5 + 2\sqrt{3}} (\mathrm{Im}(z)(3 + 2\sqrt{3}) - \mathrm{Re}(z)(2 + 2\sqrt{3})) \le \\  \dfrac{1}{5 + 2\sqrt{3}} ((3 + 2\sqrt{3}) - (1 + \sqrt{3})(2 + 2\sqrt{3})) = 1.
		\end{multline*}
		Thus, $\sqrt{3}i/a$ and $z/a$ are not approximately collinear by Lemma~\ref{lemIfLargeImaginaryOppositeSignsThenNotApprox}. This contradiction shows that $\mathrm{Im}(z) > 1$.
		
		Now $w$ is approximately collinear either with $\sqrt{3}i$, or with $z$. Due to Lemma~\ref{lemIfLargeImaginaryOppositeSignsThenNotApprox} this implies that $\mathrm{Im}(w) > -1$.
	\end{proof}
	
	\begin{proof}[Proof of Theorem~\ref{thMain}]
		Let $E \subset \CC$ be a distinguished set. We show that $\{-1,1,\infty\} \cup E$ is splittable (by Proposition~\ref{propReductionToSharp} this will imply the statement of Theorem~\ref{thMain}). Due to Proposition~\ref{propIfNotSqrt3ThenSplittable} if $\pm \sqrt{3}i \notin E$, then this statement is true. Thus, it remains to consider the case, when $\sqrt{3}i \in E$ (the other case, i.e. $-i\sqrt{3} \in E$ is handled similarly). Let $z$ and $w$ denote the remaining points of $E$, i.e. $E = \{\sqrt{3}i, z, w\}$. Now assume the contrary, i.e. $\{-1, 1, \infty\} \cup E$ is not splittable. Then by Proposition~\ref{propReductionToSharp} $E$ is not splittable by a strip, so by Lemma~\ref{lemIfSqrt3AndNotSplittableByAStripThenHalfPlane} we have inequalities $\mathrm{Im}(z), \mathrm{Im}(w) > -1$. Now we introduce the mapping $S$ defined as $S(x) = e^{i2\pi/3}x + e^{i\pi/3}$ (this mapping is the rotation with center $i/\sqrt{3}$ by the angle $2\pi/3$). It is easy to verify that $S(1) = \sqrt{3}i$, $S(\sqrt{3}i) = -1$, and $S(-1) = 1$. Moreover, consider the sets $E_1 = \{\sqrt{3}i, S(z), S(w)\}$ and  $E_2 = \{\sqrt{3}i, S^2(z), S^2(w)\}$ (also let $E_0 = E$). Clearly, by Lemma~\ref{lemProjectiveInvarianceSplittable}, if any of the sets $E_j \cup \{1, -1, \infty\}$, $j = 0,1,2$ is splittable, then all of them are splittable. Thus, all the sets $E_j$ for $j =1,2,3$ are not splittable by a strip. Therefore, $z,w \in T$, where $T$ is defined as
		\[
		T = \{x \in \CC: \mathrm{Im}(x) > -1, \mathrm{Im}(S(x)) > -1, \mathrm{Im}(S^2(x)) > -1\}.
		\]
		The set $T$ is easily seen (see Fig.~\ref{figTriangleT}) to be the open triangle
		\[
		T = \mathrm{Int} \conv\{-1 - \sqrt{3} - i, 1 + \sqrt{3} - i, (2 + \sqrt{3})i\}.
		\]
		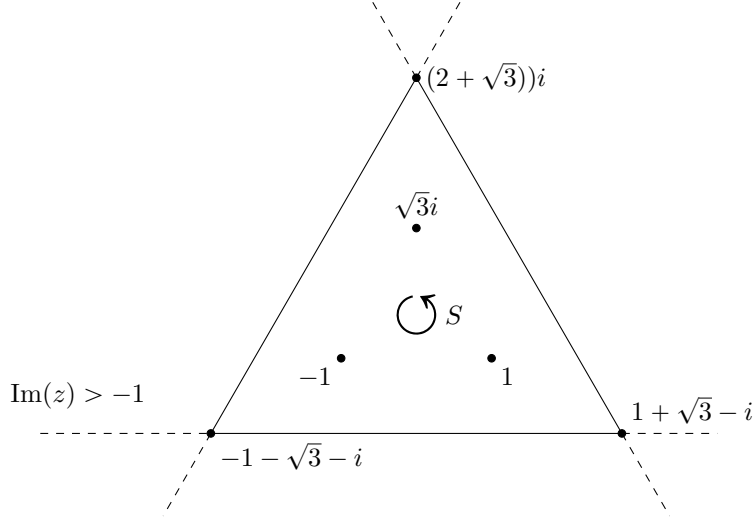
\begin{figure}
			\begin{tikzpicture}
				\filldraw (1, 0) circle[radius = 0.05] node[below right] {$1$};
				\filldraw (-1, 0) circle[radius = 0.05] node[below left] {$-1$};
				\filldraw (0, {sqrt(3)}) circle[radius = 0.05] node[above] {$\sqrt{3}i$};
				\draw[dashed] (-5, -1) -- ({-1 - sqrt(3)}, -1);
				\draw ({-1 - sqrt(3)}, -1) -- ({1 + sqrt(3)}, -1);
				\draw[dashed] ({1 + sqrt(3)}, -1) -- (4,-1);
				\draw ({-1 - sqrt(3)}, -1) -- (0, {2 + sqrt(3)});
				\draw ({1 + sqrt(3)}, -1) -- (0, {2 + sqrt(3)});
				\draw[dashed] ({1 + sqrt(3)}, -1) -- ({1 + sqrt(3) + 0.5*(3 - sqrt(3))}, {-1 - 0.5*sqrt(3)*(3 - sqrt(3))});
				\draw[dashed] ({-1 - sqrt(3)}, -1) -- ({-1 - sqrt(3) - 0.5*(3 - sqrt(3))}, {-1 - 0.5*sqrt(3)*(3 - sqrt(3))});
				\draw[dashed] (0, {2 + sqrt(3)}) -- ({0.5*(3 - sqrt(3))}, {2 + sqrt(3) + sqrt(3)*0.5*(3 - sqrt(3))});
				\draw[dashed] (0, {2 + sqrt(3)}) -- ({-0.5*(3 - sqrt(3))}, {2 + sqrt(3) + sqrt(3)*0.5*(3 - sqrt(3))});
				\filldraw ({-1-sqrt(3)}, -1) circle[radius = 0.05] node[below right] {$-1 - \sqrt{3} - i$};
				\filldraw ({1+sqrt(3)}, -1) circle[radius = 0.05] node[above right] {$1 + \sqrt{3} - i$};
				\filldraw (0, {2 + sqrt(3)}) circle[radius = 0.05] node[right] {$(2 + \sqrt{3}))i$};
				\node at (-4.5, -0.5) {$\mathrm{Im}(z) > -1$};
				\node[scale = 2] at (0, {1/sqrt(3)}) {$\circlearrowleft$};
				\node at (0.5, {1/sqrt(3)}) {$S$};
			\end{tikzpicture}
			\caption{The triangle $T$.}
			\label{figTriangleT}
		\end{figure}
		Now it remains to note that $T \subset (-1 + 2\mathbb D) \cup (1 + 2\mathbb D) \cup (\sqrt{3}i + 2\mathbb D)$, so $E$ is not distinguished. This contradiction shows that $\{-1, 1, \infty\} \cup E$ is splittable.
	\end{proof}
	
	\printbibliography

@incollection {BuchOldAndNew,
    AUTHOR = {Buchstaber, V. M. and Enolski, V. Z. and Leykin, D. V.},
     TITLE = {{$\sigma$}-functions: old and new results},
 BOOKTITLE = {Integrable systems and algebraic geometry. {V}ol. 2},
    SERIES = {London Math. Soc. Lecture Note Ser.},
    VOLUME = {459},
     PAGES = {175--214},
 PUBLISHER = {Cambridge Univ. Press, Cambridge},
      YEAR = {2020},
   MRCLASS = {14H50 (14H40 33E05)},
  MRNUMBER = {4421431},
}

@phdthesis{Smith,
	author = {Smith, Benjamin Andrew},
	title = {Explicit endomorphisms and correspondences},
	year = {2005},
	school = {The University of Sydney}
}

@article{Humbert,
  title={Sur la transformation ordinaire des fonctions abeliennes},
  author={Humbert, Georges},
  journal={Journal de Math{\'e}matiques Pures et Appliqu{\'e}es},
  volume={7},
  pages={395--417},
  year={1901}
}

@article{KleinianWeight2,
	AUTHOR = {Smirnov, Matvey},
	TITLE = {Kleinian hyperelliptic funtions of weight 2 associated with curves of genus 2},
	journal={arXiv preprint arXiv:2603.08253},
  	year={2026}
}

@article{FinalAlgorithm,
	AUTHOR = {Smirnov, Matvey},
	TITLE = {Computation of genus 2 Kleinian hyperelliptic functions via Richelot isogenies},
	journal={arXiv preprint arXiv:2603.23188},
  	year={2026}
}

@book {Cartan,
    AUTHOR = {Cartan, Henri},
     TITLE = {Elementary theory of analytic functions of one or several
              complex variables},
   EDITION = {1973},
      NOTE = {Translated from the French},
 PUBLISHER = {Dover Publications, Inc., New York},
      YEAR = {1995},
     PAGES = {228}
}

@book {ConvexLectures,
    AUTHOR = {Hug, Daniel and Weil, Wolfgang},
     TITLE = {Lectures on convex geometry},
    SERIES = {Graduate Texts in Mathematics},
    VOLUME = {286},
 PUBLISHER = {Springer, Cham},
      YEAR = {[2020] \copyright 2020},
     PAGES = {xviii+287},
}
	
\end{document}